\documentclass[12pt,leqno,oneside,letterpaper]{amsart}

\usepackage[T1]{fontenc}
\usepackage[utf8]{inputenc}
\usepackage{lmodern}
\usepackage{eulervm}

\usepackage{amsmath,amssymb,amsthm,mathtools}
\usepackage[margin=1.15in]{geometry}
\usepackage{enumitem}
\usepackage{etoolbox}
\usepackage{xurl}
\usepackage[colorlinks,linkcolor=blue,citecolor=blue]{hyperref}
\usepackage[nameinlink,noabbrev]{cleveref}
\usepackage{orcidlink}

\hypersetup{
pdftitle={Unit Indices of Shanks Orders}, pdfauthor={Junyu Lu}, pdfsubject={Units, conductors, and ideal classes of simplest cubic orders}, pdfkeywords={simplest cubic field, Shanks order, unit index, regulator, Picard group, ideal class monoid, integral conjugacy}
}

\newtheorem{theorem}{Theorem}
\newtheorem{proposition}[theorem]{Proposition}
\newtheorem{corollary}[theorem]{Corollary}
\newtheorem{lemma}[theorem]{Lemma}
\theoremstyle{remark}

\newcommand{\OO}{\mathcal O}
\newcommand{\ZZ}{\mathbb Z}
\newcommand{\QQ}{\mathbb Q}
\newcommand{\FF}{\mathbb F}
\newcommand{\Gal}{\operatorname{Gal}}
\newcommand{\Pic}{\operatorname{Pic}}
\newcommand{\Reg}{\operatorname{Reg}}
\newcommand{\Tr}{\operatorname{Tr}}
\newcommand{\Norm}{\operatorname{N}}
\newcommand{\ICM}{\operatorname{ICM}}
\newcommand{\GL}{\operatorname{GL}}
\newcommand{\SL}{\operatorname{SL}}
\newcommand{\doiurl}[1]{\href{https://doi.org/#1}{DOI}}

\title{Unit Indices of Shanks Orders}

\author[J.~Lu]{Junyu Lu\,\orcidlink{0009-0001-4973-3198}}

\address{Department of Mathematics, Yang-En University, Majia Town, Luojiang District, Quanzhou City, Fujian Province, China, 362014}

\email{jy.lu@outlook.com}


\begin{document}

\begin{abstract}
    For an integer $t\geq-1$, let $\theta_t$ be the largest real root of $g_t(X)=X^3-tX^2-(t+3)X-1$, and set $R_t=\ZZ[\theta_t]\subseteq\OO_t=\OO_{\QQ(\theta_t)}$, $N_t=[\OO_t:R_t]$, and $\varepsilon_t=[\OO_t^\times:R_t^\times]$. We determine $\varepsilon_t$ when $N_t$ is squarefree, $27$, or $343$: the only nontrivial indices in these cases are $\varepsilon_3=3$, $\varepsilon_5=7$, $\varepsilon_{12}=13$, and $\varepsilon_{54}=19$. Local conductor calculations give $\varepsilon_t\mid N_t$ for squarefree $N_t$. For arbitrary $N_t$, a regulator comparison shows that $t\geq2N_t$ implies $\varepsilon_t=1$. When $N_t=p^3$ with $p\neq3$ a rational prime, this bound and the index criterion leave at most four possible parameters with $\varepsilon_t>1$ for each fixed $p$. For arbitrary additive index, we also prove that $13\mid\varepsilon_t$ if and only if $t=12$ or $66$, with unit index $13$ in both cases. The proof determines the rational solutions of a plane quartic equation by an explicit genus-two descent and a local Chabauty argument, proving Louboutin's Conjecture~19 on its integral solutions. When $N_t$ is a rational prime, we determine the Picard kernel, the cardinality and fibers of the ideal class monoid over $\Pic(\OO_t)$, and the corresponding integral matrix-conjugacy classes.
\end{abstract}

\maketitle

\section{Introduction}\label{sec:introduction}

For every positive integer $m$, write $C_m$ for a cyclic group of order $m$. For an integer $t\geq-1$, consider the Shanks polynomial
\[
    g_t(X)=X^3-tX^2-(t+3)X-1
\]
and let $\theta_t$ be its largest real root. Set
\[
    K_t=\QQ(\theta_t),\qquad
    R_t=\ZZ[\theta_t],\qquad
    \OO_t=\OO_{K_t},\qquad
    \Delta_t=t^2+3t+9.
\]
By the rational root theorem, the only possible rational roots are $1$ and $-1$. Since $g_t(-1)=1$ and $g_t(1)=-2t-3\neq0$, the polynomial is irreducible over the rationals. The discriminant formula $\operatorname{disc}(g_t)=\Delta_t^2=(t^2+3t+9)^2>0$ shows that it has three distinct real roots. An irreducible cubic with square discriminant has Galois group $C_3$. Hence $K_t/\QQ$ is a totally real cyclic cubic extension. Shanks's study of this family \cite{Shanks1974} includes its Galois action, units, regulators, and class numbers. We restrict to $t\geq-1$ because
\[
    g_{-t-3}(-X-1)=-g_t(X).
\]
Distinct parameters in this range can still define the same field.

Thomas's unit theorem \cite[Theorem~(3.10), p.~40]{Thomas1979} states that $\theta_t$ and $1+\theta_t$ form a fundamental system of units for the equation order $R_t$ for every $t\geq-1$.\footnote{We clarify the sign convention in Thomas's proof and check its endpoint $t=-1$ at the start of Section~\ref{sec:unit-proof}.} Thus
\[
    R_t^\times=\langle-1,\theta_t,1+\theta_t\rangle.
\]
The order $R_t$ need not be maximal, so $R_t^\times$ need not equal $\OO_t^\times$. We study the finite index
\[
    \varepsilon_t=[\OO_t^\times:R_t^\times].
\]
Louboutin studied this index for all $t\geq-1$. He uses
\[
    P_t(X)=X^3+tX^2-(t+3)X+1=-g_t(-X)
\]
and takes $\xi_t=-\theta_t$. Therefore his unit index
\[
    j_t=
        [\OO_t^\times:\langle-1,\xi_t,\xi_t-1\rangle]
\]
equals $\varepsilon_t$. A PARI/GP computation reported by Louboutin \cite[Section~3, p.~72]{Louboutin2020} for $-1\leq t\leq4\cdot10^4$ yielded seven parameters with nontrivial unit index:
\[
    \begin{array}{c|rrrrrrr}
        t             & 3 & 5 & 12 & 54 & 66 & 1259 & 2389 \\ \hline
        \varepsilon_t & 3 & 7 & 13 & 19 & 13 & 97   & 31
    \end{array}
\]
Louboutin \cite[p.~73]{Louboutin2020} asked whether these are the only parameters with $\varepsilon_t>1$.

We provide a partial answer to his question. Our first main result treats the case when the additive index
\[
    N_t=[\OO_t:R_t]
\]
is squarefree.

\begin{theorem}\label{thm:squarefree-unit-rigidity}
    Let $t\in\ZZ$ with $t\geq -1$, and suppose that $N_t$ is squarefree, with $1$ regarded as squarefree. Then
    \[
        \varepsilon_t=
        \begin{cases}
            3, & t=3,              \\
            7, & t=5,              \\
            1, & \text{otherwise}.
        \end{cases}
    \]
\end{theorem}

\begin{corollary}\label{cor:prime-index-unit-rigidity}
    Let $t\in\ZZ$ with $t\geq -1$, and suppose that $N_t$ is a rational prime. Then
    \[
        \varepsilon_t=
        \begin{cases}
            3, & t=3,              \\
            7, & t=5,              \\
            1, & \text{otherwise}.
        \end{cases}
    \]
\end{corollary}

\noindent We also settle the case $N_t=27$, which includes the exceptional parameter $t=12$ in Louboutin's table.

\begin{theorem}\label{thm:index27}
    Let $t\in\ZZ$ with $t\geq -1$, and suppose that $N_t=27$. Then
    \[
        \varepsilon_t=
        \begin{cases}
            13, & t=12,             \\
            1,  & \text{otherwise}.
        \end{cases}
    \]
\end{theorem}

\noindent We next determine all unit indices when $N_t=7^3=343$.

\begin{theorem}\label{thm:index343}
    Let $t\in\ZZ$ with $t\geq-1$, and suppose that $N_t=343$. Then
    \[
        \varepsilon_t=
        \begin{cases}
            19, & t=54,             \\
            1,  & \text{otherwise}.
        \end{cases}
    \]
\end{theorem}

\noindent These theorems recover the entries $t=3,5,12,54$ from Louboutin's table.

Our next result treats divisibility of the unit index by $13$ without restricting the additive index $N_t$.

\begin{theorem}\label{app:unit13}
    Let $t\in\ZZ$ with $t\geq-1$. Then $13\mid\varepsilon_t$ if and only if $t\in\{12,66\}$. In both exceptional cases, $\varepsilon_t=13$.
\end{theorem}

The key input is the following rational-point theorem. It strengthens Louboutin's Conjecture~19 \cite[Conjecture~19, p.~78]{Louboutin2020}, which asks for the integral points on the same quartic, to a statement over $\QQ$.

\begin{theorem}\label{thm:louboutin19}
    Put
    \[
        \Phi(A,B)=A^3-B^4-4AB^2-2A^2+3AB+A+4B-3.
    \]
    The only pairs $(A,B)\in\QQ^2$ satisfying $\Phi(A,B)=0$ are
    \[
        (A,B)\in\{(0,1),(2,1),(4,-1)\}.
    \]
    In particular, these are the only integral solutions.
\end{theorem}

Let $c_t$ be the positive conductor of the abelian extension $K_t/\QQ$. Kashio and Sekigawa record the identity \cite[p.~291]{KashioSekigawa2021}
\[
    c_t=\sqrt{\operatorname{disc}(\OO_t)}
\]
and their conductor formula \cite[equation~(5), p.~294]{KashioSekigawa2021} gives an explicit expression for $c_t$. Since $\operatorname{disc}(R_t)=\Delta_t^2$, comparison of discriminants gives
\[
    N_t=[\OO_t:R_t]=\frac{\Delta_t}{c_t}\,.
\]
Their valuation formulas \cite[Propositions~3.1--3.3]{KashioSekigawa2021} determine $v_p(N_t)$ for every rational prime $p$. If $N_t$ is squarefree and $p\mid N_t$, we prove that $p$ is totally ramified. Write $\mathfrak P_p$ for the unique prime ideal of $\OO_t$ above $p$. Then
\[
    R_t\otimes\ZZ_p=\ZZ_p+\mathfrak P_p^2(\OO_t\otimes\ZZ_p),\qquad
    (R_t:\OO_t)\otimes\ZZ_p=\mathfrak P_p^2(\OO_t\otimes\ZZ_p).
\]
The change-of-order exact sequence of Kl\"uners and Pauli \cite[Theorem~5.6 and Proposition~6.2]{KluenersPauli2005} then embeds $\OO_t^\times/R_t^\times$ into $\prod_{p\mid N_t}(\FF_p,+)$. It follows that $\varepsilon_t\mid N_t$.

We combine this divisibility with regulator estimates and congruences. A separate cubic-power argument treats the rational prime $3$. Louboutin's Theorems~17 and~18 \cite[Theorems~17 and~18]{Louboutin2020} state that $3\mid\varepsilon_t$ if and only if $t=3$, and that $7\mid\varepsilon_t$ if and only if $t=5$, for all $t\geq-1$. Our proofs are independent of those two theorems. We treat divisibility by $3$ for all $t$ and divisibility by $7$ under the squarefree-index hypothesis.

When $N_t=27$, Hoshi and Iida's criterion \cite[Corollary~1.2(3), p.~2]{HoshiIida2026} gives $t=27n+12$ for an integer $n\geq0$ and requires $\Delta_t/27$ to be squarefree. We give an integral basis, prove that the order conductor is $9\OO_t$, and compute the residue-unit quotient. A regulator comparison rules out $\varepsilon_t=13$ for $t\geq39$, while an explicit change of free-unit basis gives $\varepsilon_{12}=13$.

For $N_t=p^3$, where $p\neq3$ is a rational prime, Kashio--Sekigawa's power-basis formula \cite[Theorem~1.1, formula~(3), and Remark~1.2, p.~291]{KashioSekigawa2021} specializes to
\[
    \OO_t=\ZZ[(\theta_t-a)/p],
    \qquad 3a\equiv t\pmod p.
\]
We use this basis to prove that the order conductor is $p^2\OO_t$. We also compute the order of the residue-unit quotient in every case and its group structure when $p$ is unramified.

For arbitrary additive index, a regulator comparison gives $\varepsilon_t=1$ when $t\geq2N_t$. In the prime-cube case, this bound, the condition $p^3\mid\Delta_t$, and Hensel lifting leave at most four possible exceptional parameters for each fixed $p$; see \Cref{cor:prime-cube-finite}.

For $p=7$, the remaining arguments prove \Cref{thm:index343}. Hoshi's theorem \cite[Theorem~1.4]{Hoshi2011} already gives $K_{54}=K_0$; we give explicit unit coordinates to obtain $\varepsilon_{54}=19$.

The unit results of Gil-Mu\~noz and Tinkov\'a \cite[Proposition~3.1 and Corollary~3.4]{GilMunozTinkova2025} determine the maximal-order units in the nonmonogenic index-$3$ subfamily defined by $t>12$, $t\equiv3,21\pmod{27}$, and squarefree $\Delta_t/27$. For their root $\rho$, the relation $\rho'=-(\rho+1)/\rho$ gives $\langle-1,\rho,\rho'\rangle=\langle-1,\rho,\rho+1\rangle$. Their result therefore gives $\varepsilon_t=1$ in this subfamily, in agreement with \Cref{thm:squarefree-unit-rigidity}. Our proof is independent.

For an order $S$ in $K_t$, we write $\Pic(S)$ for the group of invertible fractional $S$-ideal classes. We write $\ICM(S)$ for the monoid of full fractional $S$-ideals modulo $S$-module isomorphism. Marseglia's equivalence \cite[Corollary~3.4]{Marseglia2020} identifies these classes with classes modulo $K_t^\times$-homothety. Suppose now that $N_t$ is a rational prime. The local calculation and the change-of-order exact sequence \cite[Theorem~5.6 and Proposition~6.2]{KluenersPauli2005} give
\[
    \frac{\#\Pic(R_t)}{\#\Pic(\OO_t)}
    =\frac{N_t}{[\OO_t^\times:R_t^\times]}.
\]
The order class-number formula \cite[Theorem~5.3(b), p.~17]{ArpinMarsegliaSpringer2025} includes this equality as a special case. Section~\ref{sec:ideal-classes} identifies the Picard kernel and every fiber explicitly. Section~\ref{sec:matrix-conjugacy} translates these fibers into integral matrix-conjugacy classes.

The proofs are organized around unit indices and their applications. Sections~\ref{sec:local-structure} and~\ref{sec:unit-proof} establish the local structure and the general restrictions used in the squarefree case. Sections~\ref{sec:index27} and~\ref{sec:prime-cube} treat the prime-cube cases. Section~\ref{app:thirteenth} then proves Theorems~\ref{app:unit13} and~\ref{thm:louboutin19} by an explicit genus-two argument. The final two sections apply the unit-index results to ideal classes and integral matrix conjugacy.

\section{Local structure at a squarefree additive index}
\label{sec:local-structure}

\begin{proposition}\label{prop:squarefree-local-structure}
    Assume that $N_t$ is squarefree, and let $p$ be a rational prime dividing $N_t$. Then $p$ is totally ramified in $K_t$, so there is a unique prime ideal $\mathfrak P_p$ of $\OO_t$ such that $p\OO_t=\mathfrak P_p^3$. Set
    \[
        A_p=\OO_t\otimes_{\ZZ}\ZZ_p,\qquad S_p=R_t\otimes_{\ZZ}\ZZ_p.
    \]
    Then
    \[
        S_p=\ZZ_p+\mathfrak P_p^2A_p,\qquad
        \mathfrak f_tA_p=\mathfrak P_p^2A_p,\qquad
        \frac{A_p^\times}{S_p^\times}\cong(\FF_p,+),
    \]
    where $\mathfrak f_t=(R_t:\OO_t)$. Consequently,
    \[
        \OO_t^\times/R_t^\times\lhook\joinrel\longrightarrow\prod_{p\mid N_t}(\FF_p,+),
        \qquad \varepsilon_t\mid N_t.
    \]
\end{proposition}

\begin{proof}
    Write $e_{\Delta}=v_p(\Delta_t)$ and $b_c=v_p(c_t)$. The squarefree-index hypothesis gives $v_p(N_t)=e_{\Delta}-b_c=1$. We treat the prime $3$ separately. Kashio--Sekigawa's Propositions~3.1(i) and~3.2 give $e_{\Delta}\in\{0,2,3\}$ and $b_c\in\{0,2\}$. Thus $e_{\Delta}-b_c=1$ forces $(e_{\Delta},b_c)=(3,2)$.

    For $p\neq3$, their Propositions~3.1 and~3.3 give $b_c\in\{0,1\}$, and $b_c=0$ only if $e_{\Delta}\equiv0\pmod3$. If $b_c=0$, then $e_{\Delta}-b_c=1$ would give $e_{\Delta}=1$, a contradiction. Therefore $(e_{\Delta},b_c)=(2,1)$. In summary,
    \begin{equation}
        (e_{\Delta},b_c)=
        \begin{cases}
            (3,2), & p=3,    \\
            (2,1), & p\neq3.
        \end{cases}
        \label{eq:valuation-pairs}
    \end{equation}
    Kashio--Sekigawa's conductor formula and valuation results \cite[equation~(5), p.~294 and Propositions~3.1--3.3, pp.~298--299]{KashioSekigawa2021} give these cases. In particular, $p\mid c_t$. Since $\operatorname{disc}(\OO_t)=c_t^2$, the rational prime $p$ divides the field discriminant and therefore ramifies in $K_t$.

    Let $r_p$ be the number of prime ideals above $p$. Their ramification indices and residue degrees have common values $e_{\mathrm{ram}}$ and $f_{\mathrm{res}}$ because $K_t/\QQ$ is Galois. The degree formula
    \[
        3=r_pe_{\mathrm{ram}}f_{\mathrm{res}}
    \]
    and the inequality $e_{\mathrm{ram}}>1$ force $e_{\mathrm{ram}}=3$ and $r_p=f_{\mathrm{res}}=1$. Thus $p\OO_t=\mathfrak P_p^3$.

    Since $[A_p:S_p]=p^{v_p(N_t)}=p$, the additive group $A_p/S_p$ has order $p$ and is annihilated by $p$. Hence $pA_p\subseteq S_p$. Put
    \[
        \overline A=A_p/pA_p,\qquad \overline S=S_p/pA_p,\qquad \mathfrak m=\mathfrak P_pA_p/\mathfrak P_p^3A_p.
    \]
    The algebra $\overline A$ is a three-dimensional local $\FF_p$-algebra with radical $\mathfrak m$. Since $A_p$ is a discrete valuation ring with residue field $\FF_p$, both $\mathfrak m/\mathfrak m^2$ and $\mathfrak m^2$ are one-dimensional over $\FF_p$. The subalgebra $\overline S$ has dimension two and maps onto $\overline A/\mathfrak m\cong\FF_p$, because it contains the image of $\ZZ_p$. Hence $\ell_{\mathrm{rad}}:=\overline S\cap\mathfrak m$ has dimension one.

    If $u\in \ell_{\mathrm{rad}}\setminus\mathfrak m^2$, then $u^2$ would be a nonzero element of $\mathfrak m^2$. Indeed, a lift of $u$ to $A_p$ has normalized $\mathfrak P_p$-adic valuation one, so its square has valuation two and is nonzero modulo $\mathfrak P_p^3A_p=pA_p$. Because $\ell_{\mathrm{rad}}$ is one-dimensional and closed under multiplication, one has $u^2=\mu u$ for some $\mu\in\FF_p$. Reduction modulo $\mathfrak m^2$ forces $\mu=0$, contradicting $u^2\neq0$. Thus $\ell_{\mathrm{rad}}\subseteq\mathfrak m^2$, and equality follows from the dimensions. Taking inverse images gives $S_p=\ZZ_p+\mathfrak P_p^2A_p$.

    This formula also gives $\mathfrak P_pA_p\cap S_p=\mathfrak P_p^2A_p$. Indeed, an element of $\ZZ_p$ that lies in $\mathfrak P_pA_p$ is divisible by $p$, and $pA_p=\mathfrak P_p^3A_p\subseteq\mathfrak P_p^2A_p$. Since $\mathfrak f_t=\operatorname{Ann}_{R_t}(\OO_t/R_t)$ and $\OO_t/R_t$ is finite, localization followed by $p$-adic completion gives
    \[
        \mathfrak f_tA_p
        =\operatorname{Ann}_{S_p}(A_p/S_p)
        =\{x\in A_p:xA_p\subseteq S_p\}.
    \]
    The ideal $\mathfrak P_p^2A_p$ is contained in this local conductor. The only larger $A_p$-ideals are $\mathfrak P_pA_p$ and $A_p$, neither of which is contained in $S_p$. Hence $\mathfrak f_tA_p=\mathfrak P_p^2A_p$.

    Choose a uniformizer of $A_p$, and denote its image modulo $\mathfrak P_p^2A_p$ by $\overline{\varrho}_p$. Modulo the conductor,
    \[
        S_p/\mathfrak P_p^2A_p\cong\FF_p,\qquad A_p/\mathfrak P_p^2A_p\cong\FF_p[\overline{\varrho}_p]/(\overline{\varrho}_p^2).
    \]
    Under the second isomorphism, the image of $\ZZ_p$ is the constant subfield $\FF_p$; it is therefore the image of $S_p$ as well. Every unit in the second ring has a unique expression $a(1+b\overline{\varrho}_p)$ with $a\in\FF_p^\times$ and $b\in\FF_p$. The ideal $\mathfrak P_p^2A_p$ lies in the Jacobson radicals of both $A_p$ and $S_p$. Reduction is therefore surjective on both unit groups, with the common kernel $1+\mathfrak P_p^2A_p$. Consequently,
    \[
        A_p^\times/S_p^\times
        \cong
        (A_p/\mathfrak P_p^2A_p)^\times/
        (S_p/\mathfrak P_p^2A_p)^\times
        \cong 1+\overline{\varrho}_p\FF_p
        \cong(\FF_p,+).
    \]

    Let $\mathfrak p_p=\mathfrak P_p\cap R_t$. This is the unique prime ideal of $R_t$ above $p$. Indeed, every prime ideal of $R_t$ above $p$ lies below a prime ideal of the integral extension $\OO_t/R_t$, and $\mathfrak P_p$ is the only prime ideal of $\OO_t$ above $p$.

    The rings $A_p$ and $S_p$ are the completions of $\OO_{t,\mathfrak P_p}$ and $R_{t,\mathfrak p_p}$, respectively. Kl\"uners--Pauli's Proposition~6.2 \cite[Theorem~5.6 and Proposition~6.2]{KluenersPauli2005} identifies the local quotient by reduction modulo the conductor. The residue rings are finite, so completion does not change their unit-group quotient. Hence the factor at $\mathfrak p_p$ is
    \[
        A_p^\times/S_p^\times\cong(\FF_p,+).
    \]

    The change-of-order exact sequence embeds $\OO_t^\times/R_t^\times$ into the product of these factors. If a rational prime $\ell$ does not divide $N_t$, then the two orders agree after tensoring with $\ZZ_\ell$, so its factor is trivial. The product has order $\prod_{p\mid N_t}p=N_t$. Lagrange's theorem now gives $\varepsilon_t\mid N_t$.
\end{proof}

\begin{corollary}\label{cor:prime-local-structure}
    Assume that $N_t$ is a rational prime. Then, for the unique prime ideal $\mathfrak P$ of $\OO_t$ above $N_t$, we have
    \[
        N_t\OO_t=\mathfrak P^3,\qquad R_t=\ZZ+\mathfrak P^2,\qquad \mathfrak f_t=\mathfrak P^2,
    \]
    and
    \[
        \frac{(\OO_t/\mathfrak f_t)^\times}{(R_t/\mathfrak f_t)^\times}\cong(\FF_{N_t},+),
        \qquad \varepsilon_t\in\{1,N_t\}.
    \]
\end{corollary}

\begin{proof}
    At the rational prime $N_t$, \Cref{prop:squarefree-local-structure} gives the required identities after tensoring with $\ZZ_{N_t}$, as well as the residue-unit quotient. If $\ell\neq N_t$ is a rational prime, then
    \[
        \mathfrak P^2\otimes_\ZZ\ZZ_\ell
        =\OO_t\otimes_\ZZ\ZZ_\ell.
    \]
    Thus both $R_t$ and $\ZZ+\mathfrak P^2$ become $\OO_t\otimes_\ZZ\ZZ_\ell$ after tensoring with $\ZZ_\ell$.

    The orders $R_t$ and $\ZZ+\mathfrak P^2$ are full $\ZZ$-lattices in $K_t$, and their tensor products with $\ZZ_\ell$ agree for every rational prime $\ell$. Therefore $R_t=\ZZ+\mathfrak P^2$. The same argument gives $\mathfrak f_t=\mathfrak P^2$. The residue-unit quotient and the two possible unit indices follow from \Cref{prop:squarefree-local-structure}, since there is only one local factor.
\end{proof}

\section{General unit restrictions and squarefree indices}
\label{sec:unit-proof}

\subsection{General restrictions}

We first justify the unit basis throughout the stated parameter range. We use Thomas's theorem \cite[Theorem~(3.10)]{Thomas1979} with the constant-term convention required by its proof: if the polynomial is $X^3-e_{\mathrm{Th}}X^2-f_{\mathrm{Th}}X-k_{\mathrm{Th}}$, then $k_{\mathrm{Th}}$ is the product of the roots. The opening definition in his Section~3 \cite[p.~37]{Thomas1979} has the opposite sign. The small-case inequality in the proof of~(3.8) \cite[proof of~(3.8), pp.~39 and~43]{Thomas1979} also needs a separate check at $t=-1$.

Put $x=\theta_{-1}$ and $\upsilon=(x+1)/x$. Then
\[
    \frac65<x<\frac54,\qquad
    \upsilon=x^2+x-1,\qquad x=\upsilon^2-2.
\]
Thus $\ZZ[\upsilon]=R_{-1}$. The ordered roots of $X^3-X^2-2X+1$ are
\[
    \upsilon,\qquad \upsilon'=\frac1{x+1},\qquad \upsilon''=-x.
\]
With $k_{\mathrm{Th}}=-1$, the expression asserted to exceed $2$ on Thomas's p.~43 is instead
\[
    (\upsilon-\upsilon'')(1-\upsilon')
    =x+\frac1{x+1}<\frac74<2.
\]
Thomas's cylinder argument still applies. Its required bound $\delta'<1/2$ follows from
\[
    (\upsilon-\upsilon')(\upsilon'-\upsilon'')
    >\left(\frac95-\frac5{11}\right)
    \left(\frac49+\frac65\right)>2,
    \qquad
    \delta'=\frac1{(\upsilon-\upsilon')(\upsilon'-\upsilon'')}.
\]
The relevant cylinder consists of units $\eta_{\mathrm{Th}}$ satisfying $|\eta_{\mathrm{Th}}|,|\eta_{\mathrm{Th}}''|<1<\eta_{\mathrm{Th}}'$. Its first unit minimizes $\eta_{\mathrm{Th}}'$. Since $\upsilon^{-1}$ belongs to this cylinder, the first unit has $\eta_{\mathrm{Th}}'\leq(\upsilon')^{-1}$. Thomas's coefficient bounds restrict this first unit to $-\upsilon^2+b_{\mathrm{Th}}\upsilon+c_{\mathrm{Th}}$, where $b_{\mathrm{Th}}\in\{0,1\}$ and $c_{\mathrm{Th}}\in\ZZ$. If $b_{\mathrm{Th}}=0$, its two bounded conjugates give $\upsilon^2-1<c_{\mathrm{Th}}<x^2+1$, hence $2<c_{\mathrm{Th}}<3$, which is impossible. If $b_{\mathrm{Th}}=1$, they give
\[
    x^2+x-1<c_{\mathrm{Th}}<\upsilon^2-\upsilon+1.
\]
Since $x>6/5$ and $\upsilon<11/6$, one has $1<c_{\mathrm{Th}}<3$, so $c_{\mathrm{Th}}=2$ and the unit is $-\upsilon^2+\upsilon+2=\upsilon^{-1}$. Thomas's Galois-order argument and the unit-basis criterion used there now apply. The exponent change from $\upsilon=x^{-1}(x+1)$ and its conjugate $(x+1)^{-1}$ to $x,x+1$ has determinant $1$, giving the required basis of $R_{-1}^\times/\{\pm1\}$.

For the other small polynomial $g_0$, the positive root $y=\theta_0$ satisfies $9/5<y<2$ and the corresponding expression is $y+\frac{1}{y+1}>32/15>2$. For $t\geq1$, the inequality $g_t(2)=1-6t<0$ puts the positive root above $2$, so the earlier part of Thomas's argument applies. These checks preserve the full parameter range of the unit theorem.

We now compare the regulators of $R_t$ and $\OO_t$. For an order $S$ in a totally real cubic field, $\Reg(S)$ is the absolute value of a $2\times2$ minor of the logarithmic embedding matrix of a free-unit basis of $S^\times/\{\pm1\}$. The torsion subgroup of $S^\times$ is $\{\pm1\}$. Hence, for two orders $S\subseteq T$ with finite unit index,
\begin{equation}
    \Reg(S)=[T^\times:S^\times]\Reg(T).
    \label{eq:regulator-index}
\end{equation}

Indeed, express a free-unit basis of $S^\times/\{\pm1\}$ in a free-unit basis of $T^\times/\{\pm1\}$, and let $M_{\mathrm{inc}}$ be the resulting integral matrix. Then $|\det M_{\mathrm{inc}}|=[T^\times:S^\times]$. The same change of basis multiplies the logarithmic determinant by $|\det M_{\mathrm{inc}}|$.

\begin{lemma}\label{lem:regulator-estimates}
    If $t>1$, then
    \begin{equation}
        \Reg(R_t)
        =(\log\theta_t)^2
        -\log\theta_t\log(\theta_t+1)
        +(\log(\theta_t+1))^2,
        \label{eq:order-regulator}
    \end{equation}
    and
    \begin{equation}
        \Reg(R_t)<\log^2(t+3).
        \label{eq:order-regulator-upper}
    \end{equation}
    Moreover,
    \begin{equation}
        \Reg(\OO_t)\geq\frac1{16}\log^2\!\left(\frac{\operatorname{disc}(\OO_t)}4\right)=\frac14\log^2(c_t/2).
        \label{eq:maximal-regulator-lower}
    \end{equation}
\end{lemma}

\begin{proof}
    Descartes' rule of signs shows that $\theta_t$ is the unique positive root of $g_t$. Its other roots are $-(\theta_t+1)/\theta_t$ and $-1/(\theta_t+1)$. Thomas's unit basis therefore has logarithmic minor
    \[
        \begin{pmatrix}
            \log\theta_t                & \log(\theta_t+1) \\
            \log((\theta_t+1)/\theta_t) & -\log\theta_t
        \end{pmatrix},
    \]
    which gives \eqref{eq:order-regulator}. The inequalities
    \[
        g_t(t+1)=-2t-3<0,\qquad g_t(t+2)=t^2+3t+1>0
    \]
    imply $t+1<\theta_t<t+2$. Therefore
    \[
        0<\log\theta_t<\log(\theta_t+1)<\log(t+3).
    \]
    Equation~\eqref{eq:order-regulator} can be written as
    \[
        \Reg(R_t)
        =(\log(\theta_t+1))^2
        +\log\theta_t\bigl(\log\theta_t-\log(\theta_t+1)\bigr).
    \]
    The second term is negative, so \eqref{eq:order-regulator-upper} follows. Finally, Cusick's theorem \cite[Theorem~1, pp.~63--66]{Cusick1984Regulators} applies to the totally real cubic field $K_t$ with the same logarithmic-minor normalization and gives \eqref{eq:maximal-regulator-lower}.
\end{proof}

\begin{lemma}\label{lem:odd-unit-index}
    The index $\varepsilon_t$ is odd for every $t\geq-1$.
\end{lemma}

\begin{proof}
    Order the real embeddings so that the images of $\theta_t$ are $\theta_t>0$, $-(\theta_t+1)/\theta_t<-1$, and $-1/(\theta_t+1)\in(-1,0)$. The signatures of $-1$, $\theta_t$, and $\theta_t+1$ are then
    \[
        (-,-,-),\qquad (+,-,-),\qquad (+,-,+).
    \]
    They form a basis of $\{\pm1\}^3$, viewed as a vector space over $\FF_2$. By Thomas's theorem, a totally positive unit of $R_t$ therefore has even exponents in these three generators and is a square in $R_t^\times$.

    If $\OO_t^\times/R_t^\times$ had even order, it would contain an element of order two, represented by a unit $u\notin R_t^\times$ with $u^2\in R_t^\times$. Since $u^2$ is totally positive, it equals $v^2$ for some $v\in R_t^\times$. The identity $u^2=v^2$ in a field implies $u=\pm v$, a contradiction. Louboutin's oddness result \cite[Lemma~12, p.~73]{Louboutin2020} therefore follows directly.
\end{proof}

\begin{lemma}\label{lem:three-divisibility}
    Let $t\in\ZZ$ with $t\geq-1$. If $3\mid\varepsilon_t$, then $t=3$.
\end{lemma}

\begin{proof}
    Put $\xi=-\theta_t$ and choose the generator $\tau$ of $\Gal(K_t/\QQ)$ satisfying $\tau(\xi)=(\xi-1)/\xi$. Then $\tau^2(\xi)=-1/(\xi-1)$. This choice of $\tau$ is the inverse of the generator used in Louboutin's equation~(4) \cite[p.~72]{Louboutin2020}. Since $\xi$ and $\xi-1$ are units of $R_t$, these formulas show that $R_t$ is stable under $\tau$.

    Write $\Lambda_{\OO}=\OO_t^\times/\{\pm1\}$ and $\Lambda_R=R_t^\times/\{\pm1\}$ additively. Both are free abelian of rank two, and $[\Lambda_{\OO}:\Lambda_R]=\varepsilon_t$. The $K_t/\QQ$-norm of a unit is $\pm1$, so $(1+\tau+\tau^2)\Lambda_{\OO}=0$. Write $[\xi]\in \Lambda_R$ for the class of $\xi$. Then $[\xi],\tau [\xi]$ form a basis of $\Lambda_R$: Thomas's basis is $\xi,\xi-1$, and $\xi-1=\xi\tau(\xi)$.

    Suppose that $3\mid\varepsilon_t$. Let
    \[
        W_{\mathrm{ker}}=\ker\bigl(\Lambda_R/3\Lambda_R\longrightarrow \Lambda_{\OO}/3\Lambda_{\OO}\bigr).
    \]
    The determinant of the inclusion matrix of $\Lambda_R$ in $\Lambda_{\OO}$ is divisible by $3$, so $W_{\mathrm{ker}}\neq0$. The subspace $W_{\mathrm{ker}}$ is stable under the action $\overline\tau$ induced by $\tau$.

    The norm relation for $K_t/\QQ$ gives $1+\overline\tau+\overline\tau^2=0$ on $\Lambda_R/3\Lambda_R$. Over $\FF_3$, this identity is $(\overline\tau-1)^2=0$. If $\overline{[\xi]}$ denotes the image of $[\xi]$ in $\Lambda_R/3\Lambda_R$, then $\overline{[\xi]},\overline\tau\overline{[\xi]}$ form a basis. Hence $\overline\tau-1\neq0$, and
    \[
        \ker(\overline\tau-1)=\operatorname{im}(\overline\tau-1)
    \]
    is one-dimensional. Every nonzero $\overline\tau$-stable subspace of $\Lambda_R/3\Lambda_R$ contains this line. Indeed, if such a subspace is contained in $\ker(\overline\tau-1)$, then it equals that one-dimensional kernel. Otherwise, choose a vector in the subspace but outside the kernel. Its image under $\overline\tau-1$ is a nonzero element of both the subspace and $\operatorname{im}(\overline\tau-1)$. Therefore $(\overline\tau-1)\overline{[\xi]}\in W_{\mathrm{ker}}$.

    It follows that $(\tau-1)[\xi]\in3\Lambda_{\OO}$. Since $3[\xi]\in3\Lambda_{\OO}$, one also has $(\tau+2)[\xi]\in3\Lambda_{\OO}$. In multiplicative notation,
    \[
        \xi^2\tau(\xi)=\xi(\xi-1).
    \]
    Lifting the relation from $\Lambda_{\OO}=\OO_t^\times/\{\pm1\}$ and absorbing the possible minus sign into a cube gives
    \begin{equation}
        \eta^3=\xi(\xi-1)
        \label{eq:cube-relation}
    \end{equation}
    for some $\eta\in\OO_t^\times$. This is the power relation needed below. Louboutin's Lemma~13 \cite[pp.~73--74]{Louboutin2020} also contains this relation. All embeddings of $K_t$ are real, so uniqueness of real cube roots gives
    \[
        \tau(\eta)=-\frac{\eta}{\xi},\qquad \tau^2(\eta)=\frac{\eta}{\xi-1}.
    \]
    Therefore
    \[
        \Tr_{K_t/\QQ}(\eta^{-1})
        =\frac1\eta-\frac\xi\eta+\frac{\xi-1}{\eta}=0.
    \]
    Moreover,
    \[
        \Norm_{K_t/\QQ}(\xi)=-1,\qquad \Norm_{K_t/\QQ}(\xi-1)=1,\qquad \Norm_{K_t/\QQ}(\eta)=-1.
    \]
    The trace identity shows that $\eta\notin\QQ$, so $\eta$ has degree three. Set $\ell=\Tr_{K_t/\QQ}(\eta)\in\ZZ$. The coefficient of $X$ in the minimal polynomial of $\eta$ is
    \[
        \Norm_{K_t/\QQ}(\eta)\Tr_{K_t/\QQ}(\eta^{-1})=0.
    \]
    Hence that polynomial is $X^3-\ell X^2+1$. Newton's identities give
    \[
        \Tr_{K_t/\QQ}(\eta^3)=\ell^3-3,
    \]
    while direct reduction with the minimal polynomial of $\xi$ gives
    \[
        \Tr_{K_t/\QQ}(\xi(\xi-1))=t^2+3t+6.
    \]
    Thus
    \[
        \ell^3=\Delta_t=t^2+3t+9=\left(\frac{t+6}{3}\right)^3+\left(\frac{3-t}{3}\right)^3.
    \]
    Multiplying this identity by $27$ gives
    \[
        (3\ell)^3=(t+6)^3+(3-t)^3.
    \]
    Fermat's theorem for exponent three implies that at least one of $3\ell$, $t+6$, and $3-t$ is zero; this is also the final input in Louboutin's proof of Theorem~17 \cite[p.~77]{Louboutin2020}. Since $\ell^3=\Delta_t>0$ and $t+6>0$ for $t\geq-1$, we must have $t=3$.
\end{proof}

\begin{lemma}\label{lem:unit-index-inert-primes}
    Let $t\in\ZZ$ with $t\geq-1$. For every rational prime $\ell\equiv2\pmod3$, the integer $v_\ell(\varepsilon_t)$ is even. In particular, $\varepsilon_t\neq5$.
\end{lemma}

\begin{proof}
    Let $\Lambda_{\OO}$, $\Lambda_R$, and $\tau$ be as in the proof of \Cref{lem:three-divisibility}, and let $M_\ell$ be the $\ell$-primary part of $\Lambda_{\OO}/\Lambda_R$. The norm relation for $K_t/\QQ$ implies that $1+\tau+\tau^2$ annihilates $M_\ell$. For every $j\geq0$, the quotient $\ell^jM_\ell/\ell^{j+1}M_\ell$ is therefore a module over
    \[
        \FF_\ell[X]/(X^2+X+1).
    \]
    Since $\ell\equiv2\pmod3$, the group $\FF_\ell^\times$ has no element of order three. Also $1$ is not a root of $X^2+X+1$ in $\FF_\ell$. This polynomial is consequently irreducible, and the quotient ring is $\FF_{\ell^2}$. Each filtration quotient has even dimension over $\FF_\ell$. Their dimensions sum to $v_\ell(\#M_\ell)=v_\ell(\varepsilon_t)$, proving the claim.
\end{proof}

\begin{proposition}\label{prop:uniform-unit-bound}
    Let $t\in\ZZ$ with $t\geq-1$. If $t\geq2N_t$, then $\varepsilon_t=1$.
\end{proposition}

\begin{proof}
    The hypothesis implies $t\geq2$. Since
    \[
        \Delta_t-2N_t(t+3)=(t-2N_t)(t+3)+9>0,
    \]
    one has $c_t/2=\Delta_t/(2N_t)>t+3>1$. By \eqref{eq:regulator-index} and \Cref{lem:regulator-estimates},
    \[
        \varepsilon_t<\frac{4\log^2(t+3)}{\log^2(c_t/2)}<4.
    \]
    Since the index is odd, the only possibilities are $1$ and $3$. By \Cref{lem:three-divisibility}, the latter value requires $t=3$. The conductor formula gives $N_3=3$, so $t=3$ does not satisfy $t\geq2N_t$.
\end{proof}

\subsection{Squarefree additive indices}

\begin{proof}[Proof of \Cref{thm:squarefree-unit-rigidity}]
    The conductor formula gives $N_3=3$ and $N_5=7$. If $N_t=1$, then $t\notin\{3,5\}$, while $R_t=\OO_t$ and $\varepsilon_t=1$. Thus the assertion holds in this case.

    Assume that $N_t>1$ and $\varepsilon_t>1$. By \Cref{prop:squarefree-local-structure}, one has $\varepsilon_t\mid N_t$. Choose a rational prime $\ell$ dividing $\varepsilon_t$. Then $\ell\mid N_t$. Equation~\eqref{eq:valuation-pairs} shows that every rational prime divisor of $N_t$ also divides $c_t$. Since $N_t$ is squarefree, it follows that $N_t\mid c_t$. Thus $N_t\leq c_t$, and $\Delta_t=N_tc_t$ gives
    \begin{equation}
        c_t\geq\sqrt{\Delta_t}>t+1,
        \label{eq:conductor-lower}
    \end{equation}
    where the strict inequality follows from $\Delta_t-(t+1)^2=t+8>0$. The conductor formula gives $c_{-1}=7$, $c_0=9$, and $c_1=13$, so $N_{-1}=N_0=N_1=1$ and therefore $t\geq2$. By \Cref{lem:regulator-estimates}, equation~\eqref{eq:conductor-lower} gives
    \begin{equation}
        \ell\leq\varepsilon_t<\mathcal B(t):=
        4\left(\frac{\log(t+3)}
        {\log((t+1)/2)}\right)^2.
        \label{eq:squarefree-regulator-bound}
    \end{equation}

    For real $t>1$, put $b_{\log}(t)=\log(t+3)/\log((t+1)/2)$. Its logarithmic derivative is
    \[
        \frac{d}{dt}\log b_{\log}(t)
        =
        \frac1{(t+3)\log(t+3)}
        -
        \frac1{(t+1)\log((t+1)/2)}<0,
    \]
    because each positive factor in the second denominator is smaller than its counterpart in the first. Hence $\mathcal B$ is strictly decreasing. The exact integer inequalities $2^8\cdot13^5<11^8$ and $2^{29}\cdot25^{22}<23^{29}$ give
    \[
        \frac{\log 13}{\log(11/2)}
        <\frac85<\frac{\sqrt{11}}2,
        \qquad
        \frac{\log 25}{\log(23/2)}
        <\frac{29}{22}<\frac{\sqrt7}2.
    \]
    Therefore
    \begin{equation}
        \mathcal B(10)<4\left(\frac85\right)^2<11,
        \qquad
        \mathcal B(22)<4\left(\frac{29}{22}\right)^2<7.
        \label{eq:squarefree-cutoffs}
    \end{equation}

    If the rational prime $\ell$ satisfies $\ell\geq11$, then \eqref{eq:squarefree-regulator-bound} and \eqref{eq:squarefree-cutoffs} force $t\leq9$. Since $\ell\neq3$, equation~\eqref{eq:valuation-pairs} gives $\ell^2\mid\Delta_t$. The polynomial $t^2+3t+9$ is increasing for $t\geq2$, so
    \[
        \ell^2\leq\Delta_t\leq\Delta_9=117<121,
    \]
    a contradiction.

    If $\ell=7$, then \eqref{eq:squarefree-regulator-bound} and \eqref{eq:squarefree-cutoffs} give $2\leq t\leq21$, while \eqref{eq:valuation-pairs} gives $49\mid\Delta_t$. Write $\Delta(T)=T^2+3T+9$, so that $\Delta_t=\Delta(t)$. Its two roots modulo $7$ are $5$ and $6$, and both are simple because $\Delta'(5)\equiv6$ and $\Delta'(6)\equiv1\pmod7$.

    The first root lifts to $5$ modulo $49$, since $\Delta(5)=49$. Since $\Delta(6)/7\equiv2\pmod7$, the second root lifts to $6+7\cdot5=41$. Thus the roots modulo $49$ are $5$ and $41$, and only $t=5$ lies in the required interval.

    The case $\ell=5$ is impossible because the discriminant $-27\equiv3\pmod5$ of $\Delta$ is not a square modulo $5$. The case $\ell=2$ is impossible because $\Delta_t$ is odd. Finally, \Cref{lem:three-divisibility} shows that $\ell=3$ forces $t=3$. Thus $\varepsilon_t=1$ unless $t\in\{3,5\}$.

    For $t=3$, let $x=\theta_0$ be the largest real root of $g_0(X)=X^3-3X-1$ and set
    \[
        y=x^2+x-1=x^{-1}(x+1)^2.
    \]
    Using $x^3=3x+1$ gives
    \[
        y^2=2x^2+5x+3,\qquad y^3=12x^2+21x+4,
    \]
    so $g_3(y)=0$. Since $g_0(1)<0<g_0(2)$, one has $x>1$ and $y>0$. Descartes' rule of signs gives a unique positive root of $g_3$, so $y=\theta_3$. The same reduction gives
    \[
        x=\frac{y^2-2y-5}{3},\qquad y+1=x(x+1).
    \]
    Hence $K_3=K_0$ and therefore $\OO_3=\OO_0$. Since $c_0=\Delta_0=9$, one has $R_0=\OO_0$. Thomas's theorem makes $(x,x+1)$ a free-unit basis of $\OO_0^\times/\{\pm1\}$ and $(y,y+1)$ a free-unit basis of $R_3^\times/\{\pm1\}$. Relative to the first basis, the exponent vectors of the second are the columns of
    \[
        \begin{pmatrix}-1&1\\2&1\end{pmatrix}.
    \]
    Its determinant has absolute value $3$, so \eqref{eq:regulator-index} gives $\varepsilon_3=3$. Discriminant comparison gives $N_3=27/9=3$.

    For $t=5$, let $x=\theta_{-1}$ be the largest real root of $g_{-1}(X)=X^3+X^2-2X-1$ and put
    \[
        y=x^2+3x+1=x(x+1)^2.
    \]
    Using the relation $x^3=-x^2+2x+1$ gives
    \[
        y^2=8x^2+17x+6,\qquad
        y^3=48x^2+109x+39.
    \]
    These identities imply $g_5(y)=0$. Since $g_{-1}(1)<0<g_{-1}(2)$, one has $x>1$ and hence $y>0$. Descartes' rule of signs gives a unique positive root of $g_5$, so $y=\theta_5$. The same reduction gives
    \[
        x=\frac{-y^2+8y-2}{7},\qquad
        y+1=x^{-2}(x+1)^3.
    \]
    Thus $K_5=K_{-1}$ and $\OO_5=\OO_{-1}$. Hence $c_5=c_{-1}=\Delta_{-1}=7$ and $R_{-1}=\OO_{-1}$. Thomas's theorem makes $(x,x+1)$ a free-unit basis of $\OO_{-1}^\times/\{\pm1\}$ and $(y,y+1)$ a free-unit basis of $R_5^\times/\{\pm1\}$. Relative to the first basis, the exponent vectors of the second are the columns of
    \[
        \begin{pmatrix}1&-2\\2&3\end{pmatrix}.
    \]
    The absolute value of the determinant is $7$, so \eqref{eq:regulator-index} gives $\varepsilon_5=7$. Also
    \[
        N_5=[\OO_5:R_5]=\frac{\Delta_5}{\Delta_{-1}}=\frac{49}{7}=7.
    \]
\end{proof}

\begin{proof}[Proof of \Cref{cor:prime-index-unit-rigidity}]
    Every rational prime is squarefree, so the assertion follows from \Cref{thm:squarefree-unit-rigidity} together with $N_3=3$ and $N_5=7$.
\end{proof}

\section{\texorpdfstring{The case $N_t=27$}{The case N(t) = 27}}\label{sec:index27}

Hoshi and Iida's index criterion \cite[Corollary~1.2(3), p.~2]{HoshiIida2026} states that $N_t=27$ if and only if there is an integer $n$ such that $t=27n+12$ and the integer
\begin{equation}
    d_t:=\frac{\Delta_t}{27}=27n^2+27n+7
    \label{eq:index27-criterion}
\end{equation}
is squarefree. Since $t\geq-1$, the relation $t=27n+12$ forces $n\geq0$. In this case $c_t=d_t$.

\begin{proposition}\label{prop:index27-local-structure}
    Assume that $N_t=27$, write $t=27n+12$, and set $\beta=(\theta_t-1)/3$. Then
    \begin{equation}
        \OO_t=\ZZ[\beta],\qquad R_t=\ZZ+3\ZZ\beta+9\ZZ\beta^2,\qquad \mathfrak f_t=9\OO_t.
        \label{eq:index27-orders}
    \end{equation}
    The rational prime $3$ is unramified in $K_t$ and therefore either splits completely or remains inert. Moreover,
    \begin{equation}
        \frac{(\OO_t/9\OO_t)^\times}{(R_t/9\OO_t)^\times}
        \cong
        \begin{cases}
            C_3\times C_2^2,  & 3\text{ splits completely in }K_t, \\
            C_3\times C_{13}, & 3\text{ is inert in }K_t.
        \end{cases}
        \label{eq:index27-residue-quotient}
    \end{equation}
\end{proposition}

\begin{proof}
    Substitution of $\theta_t=3\beta+1$ into $g_t(\theta_t)=0$ gives the monic polynomial
    \begin{equation}
        h_n(Y)=Y^3-(9n+3)Y^2-(9n+4)Y-(2n+1)
        \label{eq:index27-polynomial}
    \end{equation}
    for $\beta$, so $\beta$ is integral. The coordinate matrix of the $\ZZ$-basis $(1,\theta_t,\theta_t^2)$ of $R_t$ relative to the $\ZZ$-basis $(1,\beta,\beta^2)$ of $\ZZ[\beta]$ has determinant $27$. Hence $[\ZZ[\beta]:R_t]=27$ and
    \[
        \operatorname{disc}(\ZZ[\beta])=\frac{\Delta_t^2}{27^2}=d_t^2=\operatorname{disc}(\OO_t).
    \]
    Since $\ZZ[\beta]\subseteq\OO_t$, equality of discriminants gives $\OO_t=\ZZ[\beta]$, and the same coordinate change gives the formula for $R_t$ in \eqref{eq:index27-orders}.

    The ideal $9\OO_t$ is contained in $R_t$, so $9\OO_t\subseteq\mathfrak f_t$. Conversely, write $z=r+3s\beta+9w\beta^2\in\mathfrak f_t$. From \eqref{eq:index27-polynomial},
    \[
        z\beta
        =9w(2n+1)+\bigl(r+9w(9n+4)\bigr)\beta
        +\bigl(3s+9w(9n+3)\bigr)\beta^2.
    \]
    The coefficient conditions in \eqref{eq:index27-orders}, applied to $z\beta\in R_t$, give $3\mid r$ and $3\mid s$. Moreover, the coefficient of $\beta^2$ in $z\beta^2$ is
    \[
        r+3s(9n+3)+9w\bigl((9n+3)^2+9n+4\bigr)
        \equiv r\pmod 9.
    \]
    Thus $z\beta^2\in R_t$ gives $9\mid r$. Each coefficient of $z$ in the basis $(1,\beta,\beta^2)$ is therefore divisible by $9$, so $z\in9\OO_t$. Hence $\mathfrak f_t=9\OO_t$.

    Since $d_t\equiv1\pmod3$ and $\operatorname{disc}(\OO_t)=d_t^2$, the rational prime $3$ is unramified. Because $K_t/\QQ$ is cyclic of degree three, it either splits completely or remains inert.

    More explicitly, \eqref{eq:index27-polynomial} reduces modulo $3$ to $Y^3-Y-(2n+1)$. Since $a^3-a=0$ for every $a\in\FF_3$, this polynomial has three distinct roots when $n\equiv1\pmod3$ and is an irreducible cubic otherwise. As $\OO_t=\ZZ[\beta]$, the prime $3$ therefore splits completely precisely when $n\equiv1\pmod3$, equivalently $t\equiv39\pmod{81}$; otherwise it is inert.

    Put $T_t=\ZZ+3\OO_t$. Every element of $R_t/9\OO_t$ has a unique representative $a+3b\beta$, where $a\in\ZZ/9\ZZ$ and $b\in\FF_3$. If $3\nmid a$, then this element is a unit because $a$ is a unit and $3b\beta$ is square-zero modulo $9\OO_t$. If $3\mid a$, then the whole representative is square-zero, so it is not a unit.

    Similarly, every element of $T_t/9\OO_t$ has a unique representative $a+3b\beta+3c\beta^2$, where $a\in\ZZ/9\ZZ$ and $b,c\in\FF_3$. It is a unit exactly when $3\nmid a$. Thus
    \[
        \#(R_t/9\OO_t)^\times=18,\qquad \#(T_t/9\OO_t)^\times=54,
    \]
    and $(T_t/9\OO_t)^\times/(R_t/9\OO_t)^\times\cong C_3$. The kernel of the reduction map
    \[
        (\OO_t/9\OO_t)^\times\longrightarrow(\OO_t/3\OO_t)^\times
    \]
    is $1+3\OO_t/9\OO_t$, which has order $3^3=27$. Moreover, $(T_t/9\OO_t)^\times$ is the inverse image of $\FF_3^\times$. Therefore reduction induces an isomorphism
    \[
        \frac{(\OO_t/9\OO_t)^\times}{(T_t/9\OO_t)^\times}
        \cong
        \frac{(\OO_t/3\OO_t)^\times}{\FF_3^\times}
        \cong
        \begin{cases}
            C_2^2,  & 3\text{ splits completely}, \\
            C_{13}, & 3\text{ is inert}.
        \end{cases}
    \]
    Indeed, the two residue algebras are $\FF_3^3$ and $\FF_{27}$, respectively. The inclusions $R_t\subseteq T_t\subseteq\OO_t$ give the short exact sequence
    \[
        1\longrightarrow C_3
        \longrightarrow\frac{(\OO_t/9\OO_t)^\times}{(R_t/9\OO_t)^\times}
        \longrightarrow
        \begin{cases}
            C_2^2,  & 3\text{ splits completely}, \\
            C_{13}, & 3\text{ is inert}
        \end{cases}
        \longrightarrow1.
    \]
    These groups are abelian, and the kernel and quotient have coprime orders. Their primary decompositions therefore give \eqref{eq:index27-residue-quotient}.
\end{proof}

\begin{proof}[Proof of \Cref{thm:index27}]
    The ring $R_t/9\OO_t$ is local because the representatives $a+3b\beta$ with $3\mid a$ are exactly its nonunits. Since $\mathfrak f_t=9\OO_t$, exactly one prime ideal of $R_t$ contains the conductor. Every other local factor in the change-of-order sequence is trivial.

    Because the finite quotient rings modulo $\mathfrak f_t$ are supported at the unique maximal ideal of $R_t$ containing $\mathfrak f_t$, localization at that maximal ideal does not change them or their unit groups. Kl\"uners--Pauli's conductor reduction \cite[Theorem~5.6 and Proposition~6.2]{KluenersPauli2005} therefore identifies the only nontrivial local factor with the group in \eqref{eq:index27-residue-quotient}. Thus $\OO_t^\times/R_t^\times$ embeds into that group.

    By \Cref{lem:odd-unit-index}, the index $\varepsilon_t$ is odd. The congruence $t\equiv12\pmod{27}$ in the criterion accompanying \eqref{eq:index27-criterion} and \Cref{lem:three-divisibility} show that $3\nmid\varepsilon_t$. Lagrange's theorem and \eqref{eq:index27-residue-quotient} now give
    \begin{equation}
        \varepsilon_t\in\{1,13\}.
        \label{eq:index27-unit-candidates}
    \end{equation}

    Since $n\geq0$, either $t=12$ or $t\geq39$. In the latter case,
    \[
        d_t-(t+3)=\frac{t^2-24t-72}{27}>0.
    \]
    Indeed, the numerator equals $513$ at $t=39$ and is strictly increasing for $t\geq39$. Equations~\eqref{eq:regulator-index}, \eqref{eq:order-regulator-upper}, and \eqref{eq:maximal-regulator-lower} imply
    \[
        \varepsilon_t<4\left(\frac{\log(t+3)}{\log(d_t/2)}\right)^2<4\left(\frac{\log(t+3)}{\log((t+3)/2)}\right)^2.
    \]
    The last logarithmic ratio decreases with $t$. Set $s=t+3$. Then
    \[
        \frac{\log s}{\log(s/2)}=1+\frac{\log2}{\log(s/2)}.
    \]
    At $t=39$ one has $\log42/\log21<5/4$, because this inequality is equivalent to $2^4<21$. Hence $\varepsilon_t<25/4<7$ for every $t\geq39$, and \eqref{eq:index27-unit-candidates} forces $\varepsilon_t=1$.

    It remains to treat $t=12$. Let $x=\theta_{-1}$ be the largest real root of $g_{-1}(X)=X^3+X^2-2X-1$ and set
    \[
        y=3x^2+6x+1.
    \]
    Direct reduction using $x^3=-x^2+2x+1$ gives $g_{12}(y)=0$. Since $g_{-1}(1)<0<g_{-1}(2)$, one has $x>1$ and hence $y>0$. Descartes' rule of signs shows that $g_{12}$ has exactly one positive root, so $y=\theta_{12}$. Because $g_{12}$ is irreducible, $\QQ(y)$ has degree three. As $\QQ(y)\subseteq K_{-1}$ and $[K_{-1}:\QQ]=3$, it follows that $K_{12}=\QQ(y)=K_{-1}$. The conductor formula gives $c_{-1}=\Delta_{-1}=7$, so $R_{-1}=\OO_{-1}=\OO_{12}$. The exact identities
    \[
        y=x^{-3}(x+1)^4,\qquad y+1=x(x+1)^3
    \]
    hold. Thomas's theorem makes $(x,x+1)$ a free-unit basis of $\OO_{12}^\times/\{\pm1\}$ and $(y,y+1)$ a free-unit basis of $R_{12}^\times/\{\pm1\}$. Relative to the first basis, the exponent vectors of the second are the columns of
    \[
        \begin{pmatrix}-3&1\\4&3\end{pmatrix}.
    \]
    Its determinant is $-13$, so \eqref{eq:regulator-index} gives $\varepsilon_{12}=13$.
\end{proof}

\section{\texorpdfstring{The case $N_t=p^3$ with $p\neq3$}{The case N(t) = p cubed with p not equal to 3}}\label{sec:prime-cube}

Let $p\neq3$ be a rational prime. We first give an exact criterion for $N_t=p^3$. To account for the factor at $3$, put
\[
    \nu_t=
    \begin{cases}
        1, & 3\nmid t, \\
        9, & 3\mid t.
    \end{cases}
\]

\begin{proposition}\label{prop:prime-cube-screen}
    Let $t\in\ZZ$ with $t\geq-1$, and let $p\neq3$ be a rational prime. Then $N_t=p^3$ if and only if
    \begin{equation}
        t\not\equiv3\pmod9,\quad
        p^3\mid\Delta_t,\quad
        \text{and }\ \frac{\Delta_t}{\nu_t p^3}\ \text{is squarefree}.
        \label{eq:prime-cube-screen}
    \end{equation}
    These conditions imply $p\equiv1\pmod3$ and
    \[
        \bigl(v_p(\Delta_t),v_p(c_t)\bigr)\in\{(3,0),(4,1)\}.
    \]
\end{proposition}

\begin{proof}
    Write $e_{\Delta,q}=v_q(\Delta_t)$. Kashio--Sekigawa's conductor formula \cite[Propositions~3.1--3.3]{KashioSekigawa2021} gives, for every rational prime $q\neq3$,
    \[
        v_q(N_t)=
        \begin{cases}
            e_{\Delta,q},   & e_{\Delta,q}\equiv0\pmod3,     \\
            e_{\Delta,q}-1, & e_{\Delta,q}\not\equiv0\pmod3,
        \end{cases}
    \]
    whereas $v_3(N_t)=0$ if and only if $t\not\equiv3\pmod9$. The index formulas of Hoshi and Iida \cite[Theorem~1.1 and Corollary~1.2, p.~2]{HoshiIida2026} also give these conditions. It follows that $v_p(N_t)=3$ is equivalent to $e_{\Delta,p}\in\{3,4\}$. For each rational prime $q\neq3,p$, the condition $v_q(N_t)=0$ is equivalent to $e_{\Delta,q}\in\{0,1\}$. Under $t\not\equiv3\pmod9$, the $3$-part of $\Delta_t$ is exactly $\nu_t$. Together these statements prove both directions of \eqref{eq:prime-cube-screen}. Since $c_t=\Delta_t/N_t$, they also give the stated valuation pairs.

    The integer $\Delta_t$ is odd, so $p\neq2$. Let $u=t\,3^{-1}\in\FF_p$. Then $u^2+u+1=0$. As $p\neq3$, one has $u\neq1$, so $u$ has order three in $\FF_p^\times$. Hence $p\equiv1\pmod3$ and $p\geq7$. Finally,
    \[
        (2t+3)^2=4\Delta_t-27
    \]
    shows that $p\nmid2t+3$. The two roots of $T^2+3T+9$ modulo $p$ are therefore simple and lift uniquely to roots modulo each power of $p$.
\end{proof}

\subsection{Integral bases, conductors, and residue units}

Throughout this subsection, assume $N_t=p^3$ with $p\neq3$ a rational prime, and put $A_p=\OO_t\otimes\ZZ_p$ and $S_p=R_t\otimes\ZZ_p$.

\begin{proposition}\label{prop:prime-cube-orders}
    Choose $a\in\ZZ$ with $3a\equiv t\pmod p$, and put $\gamma=(\theta_t-a)/p$. Then
    \begin{equation}
        \OO_t=\ZZ[\gamma],\qquad
        R_t=\ZZ+p\ZZ\gamma+p^2\ZZ\gamma^2,\qquad
        \mathfrak f_t=p^2\OO_t.
        \label{eq:prime-cube-orders}
    \end{equation}
    More precisely, for $\alpha=(\theta_t-t/3)/p$ in $K_t\otimes\QQ_p$, one has
    \[
        A_p=\ZZ_p\oplus\ZZ_p\alpha\oplus\ZZ_p\alpha^2,
        \qquad
        S_p=\ZZ_p\oplus p\ZZ_p\alpha\oplus p^2\ZZ_p\alpha^2.
    \]
\end{proposition}

\begin{proof}
    Kashio--Sekigawa's formula \cite[Theorem~1.1, formula~(3), and Remark~1.2, p.~291]{KashioSekigawa2021} gives the global power basis by specialization. We first verify the local basis, then compute the local conductor and pass to the global statements. Translation by $t/3$ and division by $p$ give the monic polynomial
    \begin{equation}
        h_{t,p}(U)=U^3-\frac{\Delta_t}{3p^2}U
        -\frac{(2t+3)\Delta_t}{27p^3}
        \label{eq:prime-cube-polynomial}
    \end{equation}
    for $\alpha$. By \Cref{prop:prime-cube-screen}, its coefficients lie in $\ZZ_p$. The polynomial is monic, so $\alpha$ is integral over $\ZZ_p$. The integral closure of $\ZZ_p$ in $K_t\otimes\QQ_p$ is $A_p$; hence $\ZZ_p[\alpha]\subseteq A_p$.

    The discriminant of $h_{t,p}$ is $\Delta_t^2/p^6=c_t^2$. Consequently, the discriminant ideals of $\ZZ_p[\alpha]$ and $A_p$ have the same $p$-adic valuation. Since $\alpha$ generates $K_t\otimes\QQ_p$ as a $\QQ_p$-algebra, $\ZZ_p[\alpha]$ is a full $\ZZ_p$-lattice in $A_p$. The local index--discriminant formula therefore gives $A_p=\ZZ_p[\alpha]$. Since $\theta_t=t/3+p\alpha$, changing between the bases $(1,\theta_t,\theta_t^2)$ and $(1,\alpha,\alpha^2)$ gives the asserted formula for $S_p$.

    Next, we compute the local conductor $(S_p:A_p)$. The ideal $p^2A_p$ is contained in $S_p$, so it is contained in this conductor. Conversely, write an element of the conductor as
    \[
        z=r+ps\alpha+p^2w\alpha^2,\qquad r,s,w\in\ZZ_p,
    \]
    and write $\alpha^3=b\alpha+c$ with $b,c\in\ZZ_p$, as in \eqref{eq:prime-cube-polynomial}. Then
    \[
        z\alpha=p^2wc+(r+p^2wb)\alpha+ps\alpha^2.
    \]
    Since $z\alpha\in S_p$, the coefficient conditions give $p\mid r$ and $p\mid s$. Also
    \[
        z\alpha^2=psc+(psb+p^2wc)\alpha+(r+p^2wb)\alpha^2.
    \]
    Membership in $S_p$ now gives $p^2\mid r$. Thus all three coefficients of $z$ in the basis $(1,\alpha,\alpha^2)$ are divisible by $p^2$, and $(S_p:A_p)=p^2A_p$.

    It remains to pass from the local basis to the asserted global generator $\gamma$. At $p$, the difference $\gamma-\alpha=(t/3-a)/p$ belongs to $\ZZ_p$. Since $\alpha$ is integral over $\ZZ_p$, so is $\gamma$. At every rational prime $q\neq p$, the denominator $p$ is a unit in $\ZZ_q$, so $\gamma$ is integral over $\ZZ_q$. An element of $K_t$ belongs to $\OO_t$ precisely when it is integral at every rational prime. Hence $\gamma\in\OO_t$.

    The discriminant of its power basis is $\Delta_t^2/p^6=c_t^2$. The index--discriminant formula therefore gives $\OO_t=\ZZ[\gamma]$. Finally, $\theta_t=a+p\gamma$ gives the displayed lattice for $R_t$ in \eqref{eq:prime-cube-orders}.

    At rational primes $q\neq p$, the two orders agree because $N_t=p^3$, and $p^2$ is a unit. Thus the global conductor and $p^2\OO_t$ have the same localization at every rational prime. They are equal.
\end{proof}

\begin{proposition}\label{prop:prime-cube-unit-quotient}
    There is a natural injection
    \[
        \OO_t^\times/R_t^\times\lhook\joinrel\longrightarrow
        Q_{t,p}:=\frac{(A_p/p^2A_p)^\times}{(S_p/p^2A_p)^\times}.
    \]
    The order of $Q_{t,p}$ is
    \begin{equation}
        \#Q_{t,p}=
        \begin{cases}
            p(p-1)^2,   & p\text{ splits completely in }K_t,   \\
            p(p^2+p+1), & p\text{ is inert in }K_t,            \\
            p^3,        & p\text{ is totally ramified in }K_t.
        \end{cases}
        \label{eq:prime-cube-unit-quotient}
    \end{equation}
    In the two unramified cases, respectively, one has
    \[
        Q_{t,p}\cong C_p\times C_{p-1}^2
        \qquad\text{or}\qquad
        Q_{t,p}\cong C_p\times C_{p^2+p+1}.
    \]
\end{proposition}

\begin{proof}
    By \Cref{prop:prime-cube-orders}, the conductor is $\mathfrak f_t=p^2\OO_t$. The change-of-order sequence \cite[Theorem~5.6 and Proposition~6.2]{KluenersPauli2005} therefore gives the injection. Its kernel assertion can also be checked directly: if $u\in\OO_t^\times$ reduces to a unit of $R_t/\mathfrak f_t$, then both $u$ and $u^{-1}$ are congruent modulo $\mathfrak f_t\subseteq R_t$ to elements of $R_t$. Hence $u\in R_t^\times$. All residue factors away from $p$ are trivial, and completion at $p$ does not change these finite residue rings.

    An element of $S_p/p^2A_p$ has a unique representative $r+ps\alpha$, where $r\in\ZZ/p^2\ZZ$ and $s\in\FF_p$. It is a unit exactly when $p\nmid r$: the term $ps\alpha$ is square-zero, and when $p\mid r$ the whole element is square-zero. Thus
    \[
        \#(S_p/p^2A_p)^\times=p^2(p-1).
    \]
    Reduction modulo $pA_p$ is surjective on units and has kernel $(1+pA_p)/(1+p^2A_p)\cong(pA_p/p^2A_p,+)$, of order $p^3$.

    Since $K_t/\QQ$ is cyclic of degree three, $p$ either splits completely, remains inert, or is totally ramified. In the first two cases, $A_p/pA_p$ is $\FF_p^3$ or $\FF_{p^3}$. In the ramified case, \Cref{prop:prime-cube-screen} gives $v_p(\Delta_t)=4$. The polynomial \eqref{eq:prime-cube-polynomial} is then Eisenstein: its linear coefficient has valuation two and its constant coefficient has valuation one, because $p\nmid2t+3$. Its reduction is $U^3$, so $A_p/pA_p\cong\FF_p[U]/(U^3)$. The three residue algebras have unit groups of orders $(p-1)^3$, $p^3-1$, and $p^2(p-1)$, respectively. Multiplying by $p^3$ and dividing by $p^2(p-1)$ proves \eqref{eq:prime-cube-unit-quotient}.

    To determine the unramified group structures, put $T_p=\ZZ_p+pA_p$. The group $(T_p/p^2A_p)^\times$ is the inverse image of the scalar subgroup $\FF_p^\times$ under reduction. It follows that
    \[
        1\longrightarrow
        \frac{(T_p/p^2A_p)^\times}{(S_p/p^2A_p)^\times}
        \longrightarrow Q_{t,p}
        \longrightarrow\frac{(A_p/pA_p)^\times}{\FF_p^\times}
        \longrightarrow1
    \]
    is exact. Dividing a unit of $T_p/p^2A_p$ by its constant coefficient writes it as $1+pb\alpha+pc\alpha^2$ with $b,c\in\FF_p$. Modulo $(S_p/p^2A_p)^\times$, only $c$ remains, and multiplication adds these coefficients. The kernel is therefore $C_p$. The right-hand group is $C_{p-1}^2$ in the split case and $C_{p^2+p+1}$ in the inert case. Its order is prime to $p$, so the primary decomposition of a finite abelian group gives the asserted direct products.
\end{proof}

The factor $p^2+p+1$ in the inert case is the order of $\FF_{p^3}^\times/\FF_p^\times$. In particular, $19\mid7^2+7+1=57$. Thus the local quotient can have an element of order $19$ when $p=7$. A separate argument is needed to show that a global unit represents such a class. Although $3^2+3+1=13$, the local basis construction uses $3^{-1}\in\ZZ_p$, so it does not apply when $p=3$; Section~\ref{sec:index27} supplies the separate calculation.

\subsection{\texorpdfstring{Finite reduction and the case $N_t=343$}{Finite reduction and the case of additive index 343}}

\begin{corollary}\label{cor:prime-cube-finite}
    Fix a rational prime $p\neq3$. If $N_t=p^3$ and $\varepsilon_t>1$, then $0\leq t<2p^3$ and $p\equiv1\pmod3$. For such a prime $p$, let $r_1,r_2\in\{0,\ldots,p^3-1\}$ be the two roots of $T^2+3T+9$ modulo $p^3$. Every exceptional parameter belongs to
    \[
        \{r_1,r_2,r_1+p^3,r_2+p^3\}.
    \]
    Each of these candidates must still satisfy the conditions in \eqref{eq:prime-cube-screen}.
\end{corollary}

\begin{proof}
    Since $\Delta_{-1}=7$, the index $N_{-1}$ cannot be a prime cube. The bound $t<2p^3$ follows from \Cref{prop:uniform-unit-bound}. Proposition~\ref{prop:prime-cube-screen} gives $p\equiv1\pmod3$ and two simple roots modulo $p$. Hensel lifting gives exactly two roots modulo $p^3$, each with two representatives in $[0,2p^3)$. The divisibility $p^3\mid\Delta_t$ alone need not imply $N_t=p^3$, so the other conditions in \eqref{eq:prime-cube-screen} must still be checked.
\end{proof}

\begin{corollary}\label{cor:prime-cube-certificate}
    Suppose that $N_t=p^3$ with $p\neq3$ a rational prime. If $c_t^5>32(t+3)^4$, then $\varepsilon_t=1$.
\end{corollary}

\begin{proof}
    By \Cref{prop:prime-cube-screen}, one has $p\geq7$, so $N_t=p^3\geq343$. The polynomial $\Delta_t=t^2+3t+9$ is strictly increasing for $t\geq-1$, and $\Delta_3=27$. Since $N_t\leq\Delta_t$, it follows that $t>3$. The assumed inequality gives $\log(c_t/2)>\tfrac45\log(t+3)>0$. The regulator bounds therefore give
    \[
        \varepsilon_t<\frac{4\log^2(t+3)}{\log^2(c_t/2)}<\frac{25}{4}.
    \]
    The only positive odd integers below $25/4$ are $1,3,5$. Lemma~\ref{lem:three-divisibility} excludes $3$, because that value would require $t=3$, whereas $N_3=3\neq p^3$. Lemma~\ref{lem:unit-index-inert-primes} excludes $5$.
\end{proof}

\begin{proof}[Proof of \Cref{thm:index343}]
    The two roots of $T^2+3T+9$ modulo $343$ are $54$ and $286$. Indeed, they reduce to the two simple roots $5$ and $6$ modulo $7$, and direct substitution gives divisibility by $343$. Corollary~\ref{cor:prime-cube-finite} leaves only the four candidates in the table below. None is congruent to $3$ modulo $9$, and each quotient $\Delta_t/(\nu_t\cdot343)$ is squarefree. Thus all four candidates have exact index $343$ by \Cref{prop:prime-cube-screen}, and $c_t=\Delta_t/343$ gives the conductors in the last column.
    \[
        \begin{array}{c|c|c}
            t   & \Delta_t/(\nu_t\cdot343) & c_t  \\
            \hline
            54  & 1                        & 9    \\
            286 & 241                      & 241  \\
            397 & 463                      & 463  \\
            629 & 1159=19\cdot61           & 1159
        \end{array}
    \]

    For real $t\geq286$, put $d_{343}(t)=\Delta_t/686$. At $t=286$, one has $d_{343}(t)=241/2>100$ and $t+3=289<300$, while $100^5>300^4$. Moreover,
    \[
        \frac{d}{dt}\bigl(5\log d_{343}(t)-4\log(t+3)\bigr)
        =\frac{6t^2+33t+9}{\Delta_t(t+3)}>0.
    \]
    It follows that $d_{343}(t)^5>(t+3)^4$ for every real $t\geq286$. If $N_t=343$ and $t\geq286$, then $d_{343}(t)=c_t/2$, so \Cref{cor:prime-cube-certificate} gives $\varepsilon_t=1$. This proves the claim for the three larger candidates.

    For $t=54$, let $x=\theta_0$ and set
    \[
        y=7x^2+14x+4.
    \]
    Reduction using $x^3=3x+1$ gives
    \[
        g_{54}(y)=0,\qquad
        y=x^3(x+1)^2,\qquad
        y+1=x^{-2}(x+1)^5.
    \]
    Since $x>1$, one has $y>0$. The unique positive root of $g_{54}$ is therefore $y=\theta_{54}$. Irreducibility gives $K_{54}=K_0$, in agreement with Hoshi's theorem \cite[Theorem~1.4]{Hoshi2011}. Since $R_0=\OO_0$, Thomas's theorem makes $(x,x+1)$ a free-unit basis of $\OO_{54}^\times/\{\pm1\}$. Relative to this basis, the exponent vectors of $(y,y+1)$ are the columns of
    \[
        \begin{pmatrix}3&-2\\2&5\end{pmatrix}.
    \]
    Its determinant is $19$, so $\varepsilon_{54}=19$. Discriminant comparison also gives $N_{54}=\Delta_{54}/c_0=3087/9=343$.
\end{proof}

In particular, $\varepsilon_{54}=19$ does not divide $N_{54}=343$. Thus the divisibility $\varepsilon_t\mid N_t$ in \Cref{prop:squarefree-local-structure} cannot be extended to arbitrary nonsquarefree additive indices.

\section{\texorpdfstring{The case $13\mid\varepsilon_t$}{The case 13 div epsilon}}\label{app:thirteenth}

We first determine the rational points on the normalization of Louboutin's quartic. An explicit $2$-descent shows that the Jacobian has rank one. Finite reductions then leave three residue classes, and a local calculation at $17$ shows that each contains one rational point. This proves Theorem~\ref{thm:louboutin19}; its interpretation using the trace from $K_t$ to $\QQ$ then proves Theorem~\ref{app:unit13} for every $t\geq-1$.

We give the arithmetic calculations needed for the descent and the local argument: polynomial identities, prime-ideal valuations, finite-field calculations, and congruences modulo $4$ and $17^2$. We prove the class-number assertion and give a basis of the group of unit squareclasses below. The index information needed for the reduction argument is established in Lemma~\ref{app:reduction-image}. For background, see Cassels and Flynn's genus-two framework \cite[Chapters~2, 11 and~13]{CasselsFlynn1996} and Flynn's discussion of limited index information in Chabauty arguments \cite[Comment~3.2]{Flynn1997}. We begin with the models and their changes of coordinates, including the exceptional points.

\subsection{The plane curve and its normalization}
Recall
\[
    \Phi(A,B)=A^3-B^4-4AB^2-2A^2+3AB+A+4B-3,
\]
and write
\[
    \begin{split}
        h_H(u) & =-(2u+1)(4u^5-8u^4+4u^2+2u-1)  \\
               & =-8u^6+12u^5+8u^4-8u^3-8u^2+1.
    \end{split}
\]
Let $H$ denote the smooth projective curve with affine equation $v^2=h_H(u)$. To check smoothness directly, set
\[
    \begin{split}
        A_0(u) & =960u^4-1680u^3+200u^2+224u+169, \\
        B_0(u) & =-160u^5+320u^4-134u^2-93u+14.
    \end{split}
\]
Polynomial multiplication gives
\begin{equation}\label{app:bezout}
    A_0h_H+B_0h_H'=169.
\end{equation}
Thus $h_H$ has six distinct roots, and $H$ has genus two. This identity also proves good reduction at $3$, $5$, and $17$, since these primes are odd and do not divide the leading coefficient of $h_H$.

Set $a(u)=4u^3+4u^2-1$ and $c(u)=6u^2+5u+2$. Expansion gives
\[
    \begin{split}
        \Phi(A,1+uA) & =-A^2\bigl(u^4A^2+a(u)A+c(u)\bigr), \\
        \bigl(2u^4A+a(u)\bigr)^2-h_H(u)
                     & =4u^4\bigl(u^4A^2+a(u)A+c(u)\bigr).
    \end{split}
\]
Consequently the substitution
\[
    u=\frac{B-1}{A},\qquad v=2u^4A+a(u)
\]
identifies dense open subsets of the two curves. The inverse is
\[
    A=\frac{v-a(u)}{2u^4}=-\frac{2c(u)}{v+a(u)},\qquad B=1+uA,
\]
using either expression where its denominator is nonzero.

For completeness, the affine singularities and the omitted points can be checked without elimination. On $A\ne0$ the change to $(A,u)$ is invertible. For $u\ne0$ it identifies the curve locally with the smooth equation for $H$; for $u=0$ its equation is $-A+2=0$, with nonzero derivative in $A$. On $A=0$ one has
\[
    \Phi(0,B)=-(B-1)^2(B^2+2B+3).
\]
At either root of $B^2+2B+3$, the partial derivative $\Phi_A=11B+13$ is nonzero. The remaining point $(0,1)$ is a node: with $s=B-1$, its tangent cone is $-(2A^2+5As+6s^2)$, of discriminant $-23$. In particular the node has no rational branch. Neither the line $A=0$ nor the line $B-1=0$ is a component, so the open identification also proves geometric irreducibility and identifies $H$ as the normalization of the projective plane curve. The two points above the node $(0,1)$ are
\[
    u=\frac{-5\pm\sqrt{-23}}{12},\qquad v=a(u).
\]
Indeed these are the roots of $c(u)$, and the inverse map gives $A=0$ and $B=1$ at each of them.

Write $P_\pm=(0,\pm1)$ and $W=(-1/2,0)$ on $H$. The points $P_-$ and $W$ map to $(2,1)$ and $(4,-1)$, respectively; $P_+$ maps to the plane point at infinity. At $u=\infty$, the coordinate $v/u^3$ has square $-8$, so both points above infinity are nonrational; their plane images are $(0,-1\pm\sqrt{-2})$. These observations account for all exceptional fibers. In particular, proving
\begin{equation}\label{app:H-point-target}
    H(\QQ)=\{P_-,P_+,W\}
\end{equation}
will determine the affine rational point set of $\Phi(A,B)=0$ as
\[
    \{(0,1),(2,1),(4,-1)\}.
\]
The node contributes the first point and must be included separately.

For the local calculations it is convenient to use the odd-degree model
\begin{equation}\label{app:odd-model}
    C:\quad y^2=F(x),\qquad
    F(x)=13x^5-26x^4+104x^3-208x^2+144x-32.
\end{equation}
The change of variables is
\[
    x=(u+1/2)^{-1},\qquad y=2x^3v.
\]
It sends $W$ to the unique point $\infty$, and $P_\pm$ to $(2,\pm16)$. We identify the Jacobians of these smooth projective models, and write
\[
    J=\operatorname{Jac}(C),\qquad Q=[(2,-16)-\infty].
\]

The hyperelliptic fibers give $P_-+P_+\sim2W$, and hence $[P_--P_+]=2[P_--W]=2Q$ under the identification of models. Thus this divisor class is divisible by two. The arguments below require only local information about the index of $\ZZ Q$, so we do not need to prove that $Q$ generates $J(\QQ)$.

\subsection{An explicit descent on the Jacobian}
\label{app:explicit-descent}

We prove the rank assertion by an explicit $2$-descent. Stoll's odd-degree descent \cite[Lemma~4.1 and Proposition~4.2]{Stoll2001} identifies $J(k)/2J(k)$ with a subgroup of the kernel of the norm on squareclasses. Here $k$ is $\QQ$ or one of its completions, and the norm is the determinant norm from $k[X]/(f_{\mathrm m}(X))$ to $k$, using the monic model \eqref{app:monic-descent-model} below. If this algebra is a product of fields, its norm is the product of their field norms. This identification is compatible with completion. The local dimension formula and the restriction on ramification will bound the possible squareclasses.

\begin{lemma}\label{app:descent-field}
    Let
    \[
        g_L(Z)=Z^5-Z^4+2Z^3-4Z^2+Z-1,\qquad
        L=\QQ(z),\quad g_L(z)=0.
    \]
    Then $\mathcal O_L=\ZZ[z]$, the signature of $L$ is $(1,2)$, and its class number is one.  If $w=-z^3+z^2+1$, then the squareclasses of $-1,z,w$ form a basis of $\mathcal O_L^\times/(\mathcal O_L^\times)^2$.
\end{lemma}

\begin{proof}
    For a nonzero integral ideal $\mathfrak a$ of $\mathcal O_L$, write $\Norm_{L/\QQ}(\mathfrak a)=\#(\mathcal O_L/\mathfrak a)$ for its absolute ideal norm. For an element of $L$, $\Norm_{L/\QQ}$ denotes the field norm.

    Modulo $3$, the polynomial $g_L$ has no root.  Its remainders modulo the three monic irreducible quadratics $Z^2+1$, $Z^2+Z+2$, and $Z^2+2Z+2$ are respectively $-1,-Z,1$. A reducible polynomial of degree five has a factor of degree at most two, so $g_L$ is irreducible.  The discriminant calculation gives
    \[
        \operatorname{disc}(g_L)=35152=2^4 13^3.
    \]
    At the only possible index primes, we have
    \[
        \begin{aligned}
            g_L(Z) & \equiv(Z-1)^5\pmod2,         & g_L(1) & =-2,         \\
            g_L(Z) & \equiv(Z-6)^4(Z-3)\pmod{13}, & g_L(6) & =13\cdot521.
        \end{aligned}
    \]
    Since $521\equiv1\pmod{13}$, Dedekind's index criterion excludes both $2$ and $13$ from $[\mathcal O_L:\ZZ[z]]$.  Indeed, after subtracting the indicated lifted factorization and dividing by the prime, the resulting polynomial is nonzero modulo that prime at the root of the repeated factor.  Thus $\mathcal O_L=\ZZ[z]$, and Dedekind's prime decomposition theorem gives
    \begin{equation}\label{app:descent-primes}
        (2)=\mathfrak p_{L,2}^5,\quad \Norm_{L/\QQ}(\mathfrak p_{L,2})=2;
        \qquad
        (13)=\mathfrak p_{L,13}^4\mathfrak q_{L,13},\quad
        \Norm_{L/\QQ}(\mathfrak p_{L,13})=\Norm_{L/\QQ}(\mathfrak q_{L,13})=13,
    \end{equation}
    where $z\equiv6\pmod{\mathfrak p_{L,13}}$ and $z\equiv3\pmod{\mathfrak q_{L,13}}$.

    All the terms of $g_L(Z)$ are negative when $Z<0$.  For $0\leq Z\leq1$,
    \[
        g_L(Z)=Z^4(Z-1)+2Z^2(Z-1)-2Z^2+Z-1<0.
    \]
    Also
    \[
        g_L(T+1)=T^5+4T^4+8T^3+6T^2-2
    \]
    has precisely one positive root, by Descartes' rule and the intermediate value theorem.  Therefore $L$ has signature $(1,2)$.

    Minkowski's bound is
    \[
        \frac{5!}{5^5}\left(\frac4\pi\right)^2\sqrt{35152}<12.
    \]
    The rational prime $3$ is inert in $L$. The root sets of $g_L$ modulo $5,7,11$ are respectively $\{4\},\{3\},\{8\}$.  Thus, using \eqref{app:descent-primes}, the only prime ideals of norm at most $11$ are the unique prime ideals of norms $2,5,7,11$.  The following norm identities show that each is principal:
    \[
        \begin{array}{c|r}
            a                & \Norm_{L/\QQ}(a) \\ \hline
            z-1              & 2                \\
            z^4-z^3+z^2-2z   & 5                \\
            -z^4+z^3-2z^2+3z & -7               \\
            z^2+2            & 11
        \end{array}
    \]
    Every ideal class has an integral representative of norm less than $12$, and every prime factor of such a representative is principal. Hence the class number is one.

    The identities $\Norm_{L/\QQ}(z)=\Norm_{L/\QQ}(w)=1$ show that $z,w$ are units.  Their squareclasses, together with that of $-1$, are independent. The field norm from $L$ to $\QQ$ detects the exponent of $-1$. At the prime ideal of norm $5$, the residues of $z,w$ are $4,3$, respectively; the former is a square and the latter is not. At the prime ideal of norm $7$, their residues are $3,4$, respectively; the former is not a square and the latter is. Dirichlet's unit theorem and the signature give unit rank two.  Moreover the roots of unity in $L$ are $\{\pm1\}$, since $L$ has a real embedding.  The quotient of its unit group by squares consequently has dimension three, proving the basis assertion.
\end{proof}

\begin{theorem}\label{app:rank-one}
    The Jacobian of
    \[
        H:\quad v^2=-(2u+1)(4u^5-8u^4+4u^2+2u-1)
    \]
    has rank one over $\QQ$.
\end{theorem}

\begin{proof}
    On the model \eqref{app:odd-model}, put $X=13x$ and $Y=169y$. These substitutions yield the monic model
    \begin{equation}\label{app:monic-descent-model}
        \begin{split}
            C_{\mathrm m}:\quad Y^2=f_{\mathrm m}(X),\qquad
            f_{\mathrm m}(X)={} & X^5-26X^4+1352X^3-35152X^2 \\
                                & {}+316368X-913952.
        \end{split}
    \end{equation}
    These changes of variables give an isomorphism of the smooth projective curves. Composing with the change from $H$ to $C$ sends $W=(-1/2,0)$ to the unique point $\infty$ of this odd-degree model, and $P_+=(0,1)$ to $P=(26,2704)$. Under the identification with $J$, put $E=[P-\infty]$; thus $E=-Q$.

    With $L,z,w$ as in Lemma~\ref{app:descent-field}, set
    \[
        \vartheta_L=-4z^4+6z^3+2z^2+2z+8.
    \]
    Reduction modulo $g_L$ gives $f_{\mathrm m}(\vartheta_L)=0$.  The element $\vartheta_L$ is not rational, and $[L:\QQ]=5$, so $L=\QQ(\vartheta_L)$ and $f_{\mathrm m}$ is its minimal polynomial.  Direct polynomial calculations give
    \[
        \operatorname{disc}(f_{\mathrm m})=2^{32}13^{15}.
    \]
    The odd-degree descent therefore maps
    \[
        \delta:J(\QQ)/2J(\QQ)\hookrightarrow
        \ker\bigl(\Norm_{L/\QQ}:L^\times/(L^\times)^2
        \longrightarrow\QQ^\times/(\QQ^\times)^2\bigr),
    \]
    and its $2$-Selmer group consists of classes satisfying the local image conditions. Stoll's ramification restriction \cite[Corollary~4.7]{Stoll2001} implies that such classes have even valuations away from $2$ and $13$. Stoll's point formula \cite[Lemma~4.3(1)]{Stoll2001} shows that the image of $E$ is represented by
    \[
        e=26-\vartheta_L=4z^4-6z^3-2z^2-2z+18,
        \qquad \Norm_{L/\QQ}(e)=2704^2=2^8 13^4.
    \]
    We record its valuations:
    \begin{equation}\label{app:descent-e-valuations}
        v_{\mathfrak p_{L,2}}(e)=8,\qquad
        v_{\mathfrak p_{L,13}}(e)=3,\qquad
        v_{\mathfrak q_{L,13}}(e)=1.
    \end{equation}
    For an explicit verification, put $\varpi=z-1$.  Reduction modulo $g_L$ verifies
    \begin{equation}\label{app:descent-e-unit}
        e=\varpi^8e_0,\qquad
        e_0=654z^4+280z^3+1714z^2-158z+447.
    \end{equation}
    As $e_0\equiv1\pmod{\mathfrak p_{L,2}}$ and $\Norm_{L/\QQ}(\varpi)=2$, this proves the first valuation. At $\mathfrak q_{L,13}$, the simple root $3$ of $g_L$ modulo $13$ lifts to $55$ modulo $169$: $g_L(3)=182$, $g_L'(3)=328$, and $g_L(55)\equiv0\pmod{169}$.  Moreover $e(55)\equiv117\pmod{169}$, so $v_{\mathfrak q_{L,13}}(e)=1$.  The norm identity and the residue degrees in \eqref{app:descent-primes} now give $v_{\mathfrak p_{L,13}}(e)=3$.

    Let $\alpha\in L^\times$ have field norm equal to a square in $\QQ^\times$ and even valuations at all prime ideals not lying above $2$ or $13$. Then the valuation at $\mathfrak p_{L,2}$ is even, and the valuations at $\mathfrak p_{L,13},\mathfrak q_{L,13}$ have the same parity. If this common parity is odd, put $\alpha'=\alpha/e$; otherwise, put $\alpha'=\alpha$. The element $\alpha'$ still has square field norm and has even valuations at every prime ideal. Thus $(\alpha')=\mathfrak a^2$ for a fractional ideal $\mathfrak a$. Since the class number is one, $\mathfrak a=(\beta)$ for some $\beta\in L^\times$, so $\alpha'/\beta^2$ is a unit with square field norm. Its squareclass lies in the subgroup generated by $z,w$, as shown in Lemma~\ref{app:descent-field}. Since $\alpha$ is either $\alpha'$ or $e\alpha'$, its squareclass lies in the subgroup generated by $z,w,e$. Consequently
    \begin{equation}\label{app:descent-global-bound}
        \operatorname{Sel}^{(2)}(J/\QQ)\subseteq\langle z,w,e\rangle.
    \end{equation}

    First consider the local condition at $13$. The algebra $L\otimes\QQ_{13}$ is the product of a totally ramified quartic field and $\QQ_{13}$.  There are therefore two irreducible factors of $f_{\mathrm m}$ over $\QQ_{13}$. The local dimension formula \cite[Lemma~4.4(1)]{Stoll2001} gives
    \[
        \dim_{\FF_2}J(\QQ_{13})/2J(\QQ_{13})=2-1=1.
    \]
    The image of $e$ is nontrivial by its odd valuations in \eqref{app:descent-e-valuations}, so it generates this local image.  The unit $w$ is a square in both local factors: its residues at $z=6$ and $z=3$ are $3$ and $9$, both squares in $\FF_{13}$, and Hensel's lemma applies.  The class of $z$ is nontrivial, since its residue at $\mathfrak p_{L,13}$ is the nonsquare $6$.  It cannot equal the class of $e$, because its valuations are even.  Thus the local condition at $13$ removes the factor $z$ from \eqref{app:descent-global-bound}, leaving
    \begin{equation}\label{app:descent-after-thirteen}
        \operatorname{Sel}^{(2)}(J/\QQ)\subseteq\langle w,e\rangle.
    \end{equation}

    It remains to impose the local condition at $2$.  Let $L_2=L\otimes\QQ_2$.  This is a field of degree five, and
    \[
        \varpi^5+4\varpi^4+8\varpi^3+6\varpi^2-2=0,
        \qquad \mathcal O_{L_2}=\ZZ_2[\varpi].
    \]
    The polynomial is Eisenstein, so $\varpi$ is a uniformizer and
    \[
        \mathcal O_{L_2}/2\mathcal O_{L_2}
        \simeq\FF_2[\varpi]/(\varpi^5).
    \]
    There is only one irreducible factor of $f_{\mathrm m}$ over $\QQ_2$. The same local dimension formula \cite[Lemma~4.4(1)]{Stoll2001}, with genus two, now gives
    \[
        \dim_{\FF_2}J(\QQ_2)/2J(\QQ_2)=1-1+2=2.
    \]
    In addition to the class of $e$, the local image contains the class represented by $b=3-\vartheta_L$. Indeed $f_{\mathrm m}(3)=-246575\equiv1\pmod8$ is an odd square in $\QQ_2$, so there is a $\QQ_2$-point of $C_{\mathrm m}$ with abscissa $3$.

    The local squareclass of $e$ is represented by the unit $e_0\in\mathcal O_{L_2}^{\times}$ in \eqref{app:descent-e-unit}. Expanding in $\varpi$ gives
    \[
        \begin{aligned}
            e_0={} & 654\varpi^4+2896\varpi^3+6478\varpi^2
            +6726\varpi+2937,                                   \\
            b={}   & 4\varpi^4+10\varpi^3+4\varpi^2-8\varpi-11, \\
            w={}   & 1-\varpi-2\varpi^2-\varpi^3.
        \end{aligned}
    \]
    Both $e_0,b$ are congruent to $1$ modulo $2$. A square root in $L_2$ of a unit of $\mathcal O_{L_2}$ must itself be a unit of $\mathcal O_{L_2}$. We can therefore rule out squares by reduction modulo $2$ or $4$. The square roots of $1$ in $\FF_2[\varpi]/(\varpi^5)$ are precisely $1+i\varpi^3+j\varpi^4$, for $i,j\in\{0,1\}$. Changing a lift by an element of $2\mathcal O_{L_2}$ does not change its square modulo $4$.  Reduction with the equation for $\varpi$ gives the following four possible squares modulo $4$:
    \[
        \begin{gathered}
            1,\qquad1+2\varpi,\qquad1+2\varpi^3+2\varpi^4,\qquad
            1+2\varpi+2\varpi^3+2\varpi^4.
        \end{gathered}
    \]
    But
    \[
        \begin{aligned}
            e_0  & \equiv1+2\varpi+2\varpi^2+2\varpi^4,          \\
            b    & \equiv1+2\varpi^3,                            \\
            e_0b & \equiv1+2\varpi+2\varpi^2+2\varpi^3+2\varpi^4
            \pmod{4\mathcal O_{L_2}}.
        \end{aligned}
    \]
    None belongs to the list of possible squares.  Thus the squareclasses of $e_0,b$ are independent and generate the local image, whose dimension is two.

    Every product of $e_0,b$ is congruent to $1$ modulo $2$, whereas
    \[
        w\equiv1+\varpi+\varpi^3\pmod{2\mathcal O_{L_2}}
    \]
    is not a square in $\FF_2[\varpi]/(\varpi^5)$: any square has zero coefficients in degrees one and three. Multiplying $w$ by a product of $e_0,b$ preserves this reduction. It follows that the squareclass of $w$ does not belong to the local image. The local condition at $2$ consequently reduces \eqref{app:descent-after-thirteen} to
    \[
        \operatorname{Sel}^{(2)}(J/\QQ)=\langle e\rangle.
    \]
    Equality holds because $e$ is the image of the rational class $E$ and is nontrivial at $13$.  Thus the $2$-Selmer group has dimension one. By Stoll's description of rational two-torsion \cite[Lemma~4.3(3)]{Stoll2001}, irreducibility of $f_{\mathrm m}$ gives $J(\QQ)[2]=0$. By the Mordell--Weil theorem,
    \[
        \operatorname{rank}J(\QQ)
        =\dim_{\FF_2}J(\QQ)/2J(\QQ).
    \]
    The descent injection bounds this dimension above by one, and the nontrivial image of $E$ bounds it below by one.  The rank is therefore one.
\end{proof}

\subsection{Reduction to three residue classes}
The following argument uses the finite-reduction principle of the Mordell--Weil sieve \cite[Section~2]{BruinStoll2010}, with all required data displayed.

\begin{lemma}\label{app:reduction-image}
    Let
    \[
        C:\quad y^2=F(x)=13x^5-26x^4+104x^3-208x^2+144x-32,
        \qquad J=\operatorname{Jac}(C),
    \]
    and put $Q=[(2,-16)-\infty]$.  Then $J(\QQ)$ is torsion-free and $Q$ has infinite order.  If $\operatorname{rank}J(\QQ)\leq1$, the reduction modulo $17$ of every point of $C(\QQ)$ is one of
    \[
        \infty,\qquad (2,1),\qquad (2,16).
    \]
\end{lemma}

\begin{proof}
    The polynomial $F$ has discriminant $2^{32}13^3$, so the primes $3$, $5$, and $17$ are primes of good reduction.  The following calculations establish the group orders and the reduction image.

    A reduced Mumford pair $(a,b)$ represents the class of $D-(\deg a)\infty$, where $D$ is the effective divisor defined by $a(x)=0$ and $y=b(x)$, with multiplicities determined by $a$. Thus $a$ is monic, $\deg a\leq2$, $\deg b<\deg a$, and $a\mid F-b^2$; the identity is $(1,0)$. Negation replaces $b$ by $-b$. We use the following addition rule to verify the tables below. For $(a,b)+(c,b_2)$, choose Bezout identities
    \[
        d_0=r_1a+r_2c=\gcd(a,c),\qquad
        d=s_1d_0+s_2(b+b_2)=\gcd(d_0,b+b_2).
    \]
    Form
    \[
        a_*=\frac{ac}{d^2},\qquad
        b_*=\frac{s_1r_1ab_2+s_1r_2cb+s_2(bb_2+F)}{d}\pmod {a_*}.
    \]
    If $\deg a_*>2$, replace this pair by
    \[
        \left(\operatorname{monic}\frac{F-b_*^2}{a_*},\,
        -b_*\bmod\frac{F-b_*^2}{a_*}\right)
    \]
    and repeat until its first component has degree at most two. These are Cantor's composition formulas \cite[Section~3, (C1)--(C3), p.~97]{Cantor1987} and the classical reduction step \cite[Section~4, p.~99]{Cantor1987}. See also Menezes--Wu--Zuccherato \cite[Algorithms~1--2 and Theorems~49 and~51, pp.~22--27]{MenezesWuZuccherato1996}. We explicitly normalize the first polynomial to be monic at each reduction.

    We first determine the group orders at $3$ and $5$. For each finite field used below, write $n_C(q)=\#C(\FF_q)$. Direct substitution gives $n_C(3)=7$ and $n_C(5)=6$: the lists of $F(x)$ for $x=0,\ldots,p-1$ are $(1,1,1)$ and $(3,0,1,4,3)$, respectively.  To check the corresponding counts over $\FF_{p^2}$, use $\FF_p[\jmath_p]$ with $\jmath_p^2=2$ for both primes. The following matrices give the number of solutions $y$ of $y^2=F(a+b\jmath_p)$, with rows indexed by $b$ and columns by $a$, both in increasing order:
    \[
        \begin{array}{c|c}
            p=3 & p=5 \\ \hline
            \begin{pmatrix}2&2&2\\2&0&0\\2&0&0\end{pmatrix}
                &
            \begin{pmatrix}
                2 & 1 & 2 & 2 & 2 \\
                0 & 2 & 0 & 2 & 0 \\
                0 & 0 & 0 & 2 & 2 \\
                0 & 0 & 0 & 2 & 2 \\
                0 & 2 & 0 & 2 & 0
            \end{pmatrix}
        \end{array}
    \]
    For example, a nonzero value $c_0+c_1\jmath_p$ is a square precisely when $c_0^2-2c_1^2$ is a square in $\FF_p$; this checks each matrix entry using only arithmetic modulo $p$.  Adding the point at infinity gives $n_C(9)=11$ and $n_C(25)=26$.  The genus-two point-count identity \cite[Chapter~8, equation~(8.2.7), p.~80]{CasselsFlynn1996}
    \[
        \#J(\FF_p)=\frac{n_C(p)^2+n_C(p^2)}2-p
    \]
    therefore gives $\#J(\FF_3)=27$ and $\#J(\FF_5)=26$.

    The following table lists the successive multiples of $Q$ at these two primes.  Each row follows from the preceding row by the addition rule.
    \[
        \begin{array}{c|cc|cc}
               & \multicolumn{2}{c|}{\FF_3} & \multicolumn{2}{c}{\FF_5}                   \\
            n  & a_n                        & b_n                       & a_n      & b_n  \\ \hline
            1  & x+1                        & 2                         & x+3      & 4    \\
            2  & x^2+2x+1                   & 2                         & x^2+x+4  & x+2  \\
            3  & x^2+1                      & x+1                       & x^2+x+3  & 4x+1 \\
            4  & x^2+2x                     & x+1                       & x^2+4x+1 & x+4  \\
            5  & x^2                        & 2                         & x^2+4x+2 & 4x+2 \\
            6  & x^2+2                      & 1                         & x^2+4x+4 & x    \\
            7  & x+2                        & 1                         & x^2+3x+4 & 2x   \\
            8  & x^2+2                      & x                         & x^2+2x+3 & 3x   \\
            9  & x^2+2x                     & 2                         & x^2+1    & x+4  \\
            10 & x^2+x                      & 1                         & x+2      & 2    \\
            11 & x                          & 1                         & x^2+1    & 3x+3 \\
            12 & x^2+x                      & 2x+1                      & x^2+2x+2 & x+4  \\
            13 & x^2+x+1                    & x+1                       & x+4      & 0
        \end{array}
    \]
    Since $9Q\ne0$ modulo $3$, the order of $Q$ there is $27$. Since $13Q\ne0$ and $2Q\ne0$ modulo $5$, its order there is $26$.

    Injectivity of prime-to-$p$ torsion under good reduction at $3$ and $5$ now proves $J(\QQ)_{\mathrm{tors}}=0$: the $3$-primary part is excluded at $5$, the $5$-primary part at $3$, and every other primary part by the coprime orders $27$ and $26$.  The Abel--Jacobi embedding shows that $Q\ne0$, so $Q$ has infinite order.

    Assume that the rank is at most one, and write $Q=mG$, where $G$ generates $J(\QQ)$.  If $3\mid m$, the generator $\overline Q$ of $J(\FF_3)$ would be divisible by $3$ in that cyclic group of order $27$, which is impossible.  The same argument at $5$ excludes $2\mid m$.  Consequently
    \[
        \gcd(m,6)=1.
    \]
    No assertion that $m=1$ is needed.

    We next work modulo $17$, where $Q=[(2,1)-\infty]$.  Its first nine multiples are
    \[
        \begin{array}{c|cc}
            n & a_n        & b_n    \\ \hline
            1 & x+15       & 1      \\
            2 & x^2+13x+4  & 10x+15 \\
            3 & x^2+3x+16  & 2x+8   \\
            4 & x^2+16x+6  & x+8    \\
            5 & x^2+14x+8  & 14x+11 \\
            6 & x^2+5x+8   & 9x+6   \\
            7 & x^2+11x+16 & 11x+15 \\
            8 & x^2+10x    & 2x+6   \\
            9 & x^2+16x+1  & 0
        \end{array}
    \]
    The last row is nonzero $2$-torsion, and $6Q\ne0$, so $Q$ has order $18$.  The other eight nonzero multiples are $-8Q,\ldots,-Q$. The uniqueness of reduced Mumford representations therefore gives
    \[
        \langle Q\rangle\cap\iota\bigl(C(\FF_{17})\bigr)=\{0,Q,-Q\},
        \qquad \iota(P)=[P-\infty],
    \]
    because only $Q$ and $-Q$ have first component of degree one.

    We also verify $\#J(\FF_{17})$ without a count over $\FF_{289}$. Let $R=[(0,6)-\infty]$ and $T=[(16,0)-\infty]$.  The addition rule gives
    \[
        \begin{array}{c|cc}
            n & a_n\text{ for }nR & b_n\text{ for }nR \\ \hline
            1 & x                 & 6                 \\
            2 & x^2               & 12x+6             \\
            3 & x^2+11x+2         & 11x+9             \\
            4 & x^2+12x+12        & 13x+10            \\
            5 & x^2+12x+12        & 4x+7
        \end{array}
    \]
    Thus $5R=-4R$ and $3R\ne0$, so $R$ has order $9$.  Its order-three subgroup is different from that of $\langle Q\rangle$, since the first components of $3R$ and $\pm6Q$ differ.  Hence $\langle Q,R\rangle\simeq\ZZ/18\ZZ\oplus\ZZ/9\ZZ$. The class $T=(x+1,0)$ is nonzero $2$-torsion and differs from $9Q$; it lies outside $\langle Q,R\rangle$, whose only nonzero element of order two is $9Q$.  It follows that $324$ divides $\#J(\FF_{17})$.

    The values $F(0),\ldots,F(16)$ are
    \[
        2,12,1,9,5,2,2,9,8,12,7,16,10,6,12,6,0.
    \]
    The nonzero squares modulo $17$ are $1,2,4,8,9,13,15,16$, so this list gives $\#C(\FF_{17})=18$.  The trace of Frobenius on the Jacobian over $\FF_{17}$ is therefore zero, and its characteristic polynomial is $Z^4+\kappa_2Z^2+17^2$. Each of its four Frobenius roots has complex absolute value $\sqrt{17}$ by the Weil bound.  The coefficient $\kappa_2$ is the sum of their six pairwise products, so $|\kappa_2|\leq6\cdot17$.  Hence
    \[
        \#J(\FF_{17})=17^2+1+\kappa_2\leq392<2\cdot324.
    \]
    The divisibility already proved forces $\#J(\FF_{17})=324$.

    Finally, $\gcd(m,6)=1$ makes multiplication by $m$ an automorphism of $J(\FF_{17})$.  Since $Q=mG$, the reduction of $J(\QQ)$ is exactly $\langle\overline Q\rangle$.  Its intersection with the embedded curve was determined above.  Thus every rational point reduces to $\infty,(2,1)$, or $(2,16)$, as claimed.
\end{proof}

\subsection{The local calculation at 17}

We continue with the model $C$ in \eqref{app:odd-model} and the class $Q=[(2,-16)-\infty]$.

For a divisor class $D\in J(\QQ_{17})$ and a regular differential $\omega$ on $C_{\QQ_{17}}$, write $\int_D\omega$ for the Abelian logarithm pairing \cite[Sections~4.1 and~5.1]{McCallumPoonen2012} of $D$ with the corresponding invariant differential on $J$. This pairing is additive in $D$; after a finite extension of $\QQ_{17}$, it is compatible with base change and Galois conjugation.

\begin{lemma}\label{app:tiny-integrals}
    Put \(\omega_0=dx/y\) and \(\omega_1=x\,dx/y\). Then
    \[
        \frac1{17}
        \left(\int_{18Q}\omega_0,\int_{18Q}\omega_1\right)
        \equiv (10,2)\pmod {17}.
    \]
\end{lemma}

\begin{proof}
    We use the Mumford notation from the proof of Lemma~\ref{app:reduction-image}. In particular, $Q=(x-2,-16)$. Tangent interpolation gives the exact identities
    \[
        2Q=(x^2-4x+4,32-24x),\qquad
        D_4:=4Q=(x^2-x+6,x-26).
    \]
    For the second identity one can use \(l=8-4x-2x^2-x^3\), since
    \[
        F-l^2=-(x-2)^4(x^2-x+6).
    \]

    To determine \(-14Q\), put \((a_n,b_n)=nQ\). All polynomial entries in the following two tables are reduced modulo \(289\).
    \[
        \begin{array}{c|l|l}
            n  & a_n          & b_n      \\ \hline
            2  & x^2+285x+4   & 265x+32  \\
            4  & x^2+288x+6   & x+263    \\
            6  & x^2+192x+161 & 213x+57  \\
            7  & x^2+130x+152 & 215x+117 \\
            14 & x^2+237x+210 & 135x+162
        \end{array}
    \]
    The following interpolation identities verify the preceding table.
    \[
        \begin{array}{c|l|r}
            (i,j) & l                      & \lambda \\ \hline
            (2,2) & 8+285x+287x^2+288x^3   & 288     \\
            (4,2) & 22+261x+156x^2+141x^3  & 60      \\
            (6,1) & 27+258x+77x^2          & 13      \\
            (7,7) & 196+175x+237x^2+191x^3 & 222
        \end{array}
    \]
    For each row, direct multiplication in $(\ZZ/289\ZZ)[x]$ gives
    \[
        F-l^2=\lambda a_i a_j a_{i+j},
        \qquad b_{i+j}\equiv-l\pmod {a_{i+j}}.
    \]
    Moreover, \(l\equiv b_i\pmod {a_i}\) and \(l\equiv b_j\pmod {a_j}\); when \(i=j\), the displayed factor \(a_i^2\) gives the required repeated contact. These are precisely the divisor identities for addition using the function \(y-l(x)\). The interpolation lifts over \(\ZZ_{17}\): the source polynomials in the addition rows are coprime modulo \(17\), the elements \(2b_i\) in the doubling rows are invertible modulo \(a_i\), and every \(\lambda\) is a unit modulo \(17\). Thus these tables determine the reductions of the indicated rational divisor classes.

    Negating the second Mumford coordinate now gives
    \[
        \begin{aligned}
            D_{-14} & :=-14Q=(a_{-14},b_{-14}),       \\
            a_{-14} & \equiv x^2+237x+210\pmod {289}, \\
            b_{-14} & \equiv154x+127\pmod {289}.
        \end{aligned}
    \]
    We have \(D_4-D_{-14}=18Q\), and \(D_4\) and \(D_{-14}\) have the same reduction modulo \(17\). Writing \(a_4=x^2-x+6\), we obtain
    \begin{equation}\label{app:root-displacement}
        \frac{a_{-14}-a_4}{17}\equiv-3x-5\pmod {17}.
    \end{equation}

    The discriminant of \(a_4\) is \(-23\equiv11\pmod {17}\), a nonsquare. The polynomial's roots therefore lie in the unramified quadratic extension \(E_{17}/\QQ_{17}\). Let \(\alpha_{17}\) be a root of \(a_4\), and let \(\beta_{17}\) be the root of \(a_{-14}\) with the same reduction. The root is simple modulo \(17\), so Hensel's lemma and \eqref{app:root-displacement} give
    \[
        \beta_{17}-\alpha_{17}\equiv
        -17\frac{-3\alpha_{17}-5}{2\alpha_{17}-1}\pmod {17^2}.
    \]
    The \(y\)-coordinate at \(\alpha_{17}\) is \(\alpha_{17}-26\), a unit. Consequently \(x-\alpha_{17}\) is an integral local parameter, and the matching branch at \(\beta_{17}\) has \(y=b_{-14}(\beta_{17})\). After base change to $E_{17}$, the two support points of $D_{-14}$ are matched with those of $D_4$ in the same residue discs.  Additivity identifies the integral over $D_4-D_{-14}=18Q$ with the sum of these two tiny integrals.  Galois conjugation interchanges the summands, so their sum is the trace from $E_{17}$ to $\QQ_{17}$ of either tiny integral. Reducing it modulo $17^2$ gives
    \begin{equation}\label{app:trace-integral}
        \frac1{17}\int_{18Q}\omega_i
        \equiv
        \Tr_{\FF_{17^2}/\FF_{17}}
        \left(
        \frac{\alpha_{17}^i(-3\alpha_{17}-5)}
        {(2\alpha_{17}-1)(\alpha_{17}+8)}
        \right)\pmod {17},
        \qquad i=0,1.
    \end{equation}
    Here \(\alpha_{17}\) on the right denotes its reduction. To justify the precision, write $s=x-\alpha_{17}$. The expansions of $x^i/y$, for $i=0,1$, in $s$ have integral coefficients: $2y$ is a unit, and $x=\alpha_{17}+s$ with $\alpha_{17}$ integral. For \(k\ge2\), a term of degree \(k\) in either antiderivative, evaluated at \(\beta_{17}-\alpha_{17}\in17\mathcal O_{E_{17}}\), has valuation at least \(k-v_{17}(k)\ge2\). All omitted terms therefore vanish modulo \(17^2\).

    It remains to evaluate the two traces from $\FF_{17^2}$ to $\FF_{17}$. In \(\FF_{17}[\alpha_{17}]/(\alpha_{17}^2-\alpha_{17}+6)\) we have
    \[
        (2\alpha_{17}-1)(\alpha_{17}+8)=-3,\qquad (-3)^{-1}=-6,
    \]
    and hence
    \[
        \frac{-3\alpha_{17}-5}{(2\alpha_{17}-1)(\alpha_{17}+8)}
        =\alpha_{17}-4,\qquad
        \alpha_{17}(\alpha_{17}-4)=-3\alpha_{17}-6.
    \]
    Since \(\Tr_{\FF_{17^2}/\FF_{17}}(\alpha_{17})=1\) and \(\Tr_{\FF_{17^2}/\FF_{17}}(1)=2\), the two traces are \(10\) and \(2\), respectively. Equation \eqref{app:trace-integral} proves the lemma.
\end{proof}

\begin{lemma}\label{app:rational-points}
    The rational points of \(C\) and \(H\) are
    \[
        C(\QQ)=\{\infty,(2,-16),(2,16)\},\qquad
        H(\QQ)=\{(-1/2,0),(0,-1),(0,1)\}.
    \]
\end{lemma}

\begin{proof}
    By Theorem~\ref{app:rank-one} and Lemma~\ref{app:reduction-image}, \(J(\QQ)\) has rank one, and every rational point of \(C\) reduces at \(17\) to one of \(\infty,(2,1),(2,16)\). Put \(\mathcal L_i=\int_{18Q}\omega_i\). Lemma~\ref{app:tiny-integrals} shows that \(\mathcal L_0/17\) and \(\mathcal L_1/17\) are units. The regular differential
    \[
        \omega=\omega_0-\frac{\mathcal L_0}{\mathcal L_1}\omega_1
    \]
    annihilates \(Q\), and therefore annihilates all of \(J(\QQ)\): the nonzero class \(Q\) spans \(J(\QQ)\otimes\QQ\). No assertion that \(Q\) is a generator of \(J(\QQ)\) is used. The reduction of this differential is
    \[
        \overline\omega=(1-5x)\frac{dx}{y}.
    \]

    At \(x=2\), the coefficient \(1-5x\) is \(-9\equiv8\pmod {17}\), and \(y\ne0\). Thus \(\overline\omega\) does not vanish at either finite residue class under consideration. At infinity, \(s_\infty=x^2/y\) is a uniformizer and
    \[
        x=\frac1{13s_\infty^2}+O(1),\qquad
        \frac{x\,dx}{y}=-2\,ds_\infty+O(s_\infty^2)\,ds_\infty,\qquad
        \frac{dx}{y}=O(s_\infty^2)\,ds_\infty.
    \]
    Consequently \(\overline\omega=10\,ds_\infty+O(s_\infty^2)\,ds_\infty\) there, so it does not vanish at infinity either.

    The curve has good reduction at $17$, the differential $\omega$ annihilates $J(\QQ)$, and its reduction is nonzero with vanishing order zero in each of the three remaining residue classes.  Since $0<17-2$, the local Chabauty bound \cite[Theorem~5.3(a)]{McCallumPoonen2012} gives at most one rational point in each class. These residue classes already contain \(\infty,(2,-16),(2,16)\), respectively. This determines the rational points of \(C\). The birational change of variables gives the stated rational points of \(H\).
\end{proof}

\begin{proof}[Proof of Theorem~\ref{thm:louboutin19}]
    Let $(A,B)\in\QQ^2$ satisfy $\Phi(A,B)=0$. If $A=0$, the identity
    \[
        \Phi(0,B)=-(B-1)^2(B^2+2B+3)
    \]
    forces $B=1$, because $B^2+2B+3=(B+1)^2+2$ has no rational zero. Suppose that $A\ne0$, and put
    \[
        u=\frac{B-1}{A},\qquad v=2u^4A+4u^3+4u^2-1.
    \]
    The identities in the normalization calculation give $(u,v)\in H(\QQ)$. Lemma~\ref{app:rational-points} leaves only $(0,-1)$, $(0,1)$, and $(-1/2,0)$. If $u=0$, the displayed formula forces $v=-1$ and $B=1$. Then $\Phi(A,1)=A^2(A-2)=0$ gives $A=2$. If $u=-1/2$ and $v=0$, it gives $0=A/8-1/2$, so $A=4$ and $B=-1$. Direct substitution verifies all three pairs in the statement. Since all three are integral, this also proves Louboutin's conjecture.
\end{proof}

\subsection{The unit-power criterion and the exceptional parameters}

\begin{proof}[Proof of Theorem~\ref{app:unit13}]
    We use a unit-power criterion and Theorem~\ref{thm:louboutin19} to show that $13\mid\varepsilon_t$ can occur only for $t=12$ or $t=66$. We then compute the two unit indices. Write $\xi=-\theta_t$ and take $\sigma=\tau^{-1}$, where $\tau$ is the Galois generator chosen in the proof of \Cref{lem:three-divisibility}. Thus $\sigma(\xi)=1/(1-\xi)$. As before, put
    \[
        \Lambda_{\OO}=\OO_t^\times/\{\pm1\},\qquad \Lambda_R=R_t^\times/\{\pm1\}.
    \]
    Both groups are free of rank two, and the inclusion of unit groups identifies $\Lambda_R$ with a sublattice of $\Lambda_{\OO}$. Thomas's theorem identifies the classes of $\xi,\xi-1$ as a basis of $\Lambda_R$, and $[\Lambda_{\OO}:\Lambda_R]=\varepsilon_t$. Smith normal form gives $13\mid\varepsilon_t$ if and only if
    \[
        \ker(\Lambda_R/13\Lambda_R\longrightarrow \Lambda_{\OO}/13\Lambda_{\OO})\ne0.
    \]
    The kernel is stable under $\sigma$. Modulo signs,
    \[
        \sigma(\xi)=(\xi-1)^{-1},\qquad
        \sigma(\xi-1)=\xi(\xi-1)^{-1},
    \]
    so its action on exponent columns relative to the ordered basis $(\xi,\xi-1)$ of $\Lambda_R$ is
    \[
        M_\sigma=\begin{pmatrix}0&1\\-1&-1\end{pmatrix}.
    \]
    Over $\FF_{13}$ the two eigenlines are generated by $(3,1)$ and $(1,3)$, with eigenvalues $9$ and $3$. Since the eigenvalues are distinct, every nonzero invariant subspace contains one of these eigenlines. Signs can be absorbed in thirteenth powers because $13$ is odd. For $p=13$, this gives the criterion in Louboutin's Lemma~13 \cite[pp.~73--74]{Louboutin2020}:\footnote{The prime condition in the printed lemma should be $p=3$ or $p\equiv1\pmod3$.} $13\mid\varepsilon_t$ if and only if
    \begin{equation}\label{app:power-criterion}
        \xi^3(\xi-1)=\eta^{13}\ \text{or}\
        \xi(\xi-1)^3=\eta^{13}
        \quad\text{for some }\eta\in\OO_t^\times.
    \end{equation}
    Conversely, either power relation gives a nonzero element of the kernel above.

    Suppose first that $\eta^{13}=\xi^3(\xi-1)$. Conjugating and using uniqueness of real thirteenth roots gives
    \[
        \sigma(\eta)=\eta^{-4}\xi,\qquad
        \sigma^2(\eta)=-\eta^3\xi^{-1}.
    \]
    Indeed, taking thirteenth powers on the right verifies the conjugated relations directly. Following the trace construction \cite[equations~(6)--(7), p.~75, and Conjecture~15, p.~76]{Louboutin2020}, put
    \[
        V(A,B)=A+A^{-4}B-A^3B^{-1}.
    \]
    Multiplication after clearing denominators verifies the identity
    \[
        \Phi\bigl(V(A,B),V(A^{-1},B^{-1})\bigr)
        =-\frac{(A^{13}-B^3(B-1))(A^{15}+A^8B-B^4)}{A^{12}B^4}.
    \]
    The denominators are nonzero for units. Consequently
    \[
        \Phi(\Tr_{K_t/\QQ}(\eta),\Tr_{K_t/\QQ}(\eta^{-1}))=0.
    \]
    For the second relation in \eqref{app:power-criterion}, put $\zeta=1-\xi$. Then $\sigma^2(\zeta)=1/(1-\zeta)$ and $\zeta^3(\zeta-1)=\xi(\xi-1)^3=\eta^{13}$, so the same trace argument applies.

    Theorem~\ref{thm:louboutin19} now gives
    \[
        (\Tr_{K_t/\QQ}(\eta),\Tr_{K_t/\QQ}(\eta^{-1}))\in\{(0,1),(2,1),(4,-1)\}.
    \]
    The $K_t/\QQ$-norm of $\xi$ is $-1$ and that of $\xi-1$ is $1$, so either power relation implies $\Norm_{K_t/\QQ}(\eta)=-1$. If $a=\Tr_{K_t/\QQ}(\eta)$ and $b=\Tr_{K_t/\QQ}(\eta^{-1})$, the polynomial whose roots are the three conjugates of $\eta$ is
    \[
        Z^3-aZ^2-bZ+1.
    \]
    For $(a,b)=(0,1)$ this polynomial has discriminant $-23$, whereas all conjugates of $\eta$ are real. This excludes $(0,1)$.

    Louboutin \cite[p.~78]{Louboutin2020} records the trace values and the two characteristic polynomials. We verify them directly using Newton's identities and multiplication matrices. For roots $r_1,r_2,r_3$ of $Z^3-aZ^2-bZ+1$, put $s_n=\sum r_i^n$ and $q_n=\sum r_i^{-n}$. Newton's identities give
    \[
        \begin{aligned}
            (s_0,s_1,s_2) & =(3,a,a^2+2b),                &
            s_n           & =a s_{n-1}+b s_{n-2}-s_{n-3},   \\
            (q_0,q_1,q_2) & =(3,b,b^2+2a),                &
            q_n           & =b q_{n-1}+a q_{n-2}-q_{n-3}
        \end{aligned}
    \]
    for $n\geq3$. Repeated substitution yields
    \[
        \begin{array}{c|r|r}
            (a,b)  & s_{13}   & q_{13}   \\\hline
            (2,1)  & 37221    & 2094     \\
            (4,-1) & 20487939 & -319476.
        \end{array}
    \]
    Thus the polynomial of the thirteenth powers is $Z^3-s_{13}Z^2-q_{13}Z+1$.

    The multiplication matrix of $\xi$ in the basis $(1,\xi,\xi^2)$ is
    \[
        M_\xi=\begin{pmatrix}0&0&-1\\1&0&t+3\\0&1&-t\end{pmatrix}.
    \]
    Computing the $3\times3$ characteristic determinants of the matrices
    \[
        M_\xi^3(M_\xi-I)\qquad\text{and}\qquad M_\xi(M_\xi-I)^3
    \]
    gives the respective characteristic polynomials
    \[
        \begin{split}
            Z^3- & (t^4+5t^3+17t^2+25t+21)Z^2 \\
                 & +(t^3+7t^2+22t+36)Z+1,
        \end{split}
    \]
    and
    \[
        \begin{split}
            Z^3- & (t^4+7t^3+26t^2+50t+45)Z^2 \\
                 & -(t^3+2t^2+7t-6)Z+1.
        \end{split}
    \]
    Write $B_{31}=t^3+7t^2+22t+36$ and $B_{13}=t^3+2t^2+7t-6$. Comparing the coefficients of $Z$ for the two trace possibilities shows that one of the following four expressions must vanish:
    \[
        \begin{aligned}
            B_{31}+2094   & =(t+15)(t^2-8t+142),   \\
            B_{13}-2094   & =(t-12)(t^2+14t+175),  \\
            B_{31}-319476 & =(t-66)(t^2+73t+4840), \\
            B_{13}+319476 & =(t+69)(t^2-67t+4630).
        \end{aligned}
    \]
    All four quadratic factors have negative discriminant and positive leading coefficient. The condition $t\geq-1$ therefore leaves only $t=12$ and $t=66$. At these two values the coefficients of $Z^2$ also agree with the corresponding trace values.

    It remains to verify the asserted unit indices in these cases. The proof of Theorem~\ref{thm:index27} already establishes $\varepsilon_{12}=13$ by the explicit free-unit exponent matrix
    $\left(\begin{smallmatrix}-3&1\\4&3\end{smallmatrix}\right)$.
    For $t=66$, let $z_1=\theta_1$ and set $w_1=6z_1^2+9z_1+1$. With $g_1(Z)=Z^3-Z^2-4Z-1$, the identities
    \[
        \begin{aligned}
            g_{66}(6Z^2+9Z+1)    & =(216Z^3+1188Z^2+1242Z+135)g_1(Z), \\
            Z(6Z^2+9Z+1)-(Z+1)^4 & =(1-Z)g_1(Z),                      \\
            6Z^2+9Z+2-Z^3(Z+1)   & =-(Z+2)g_1(Z)
        \end{aligned}
    \]
    give
    \[
        g_{66}(w_1)=0,\qquad
        w_1=z_1^{-1}(z_1+1)^4,\qquad
        w_1+1=z_1^3(z_1+1).
    \]
    Since $z_1>2$, the number $w_1$ is positive. Descartes' rule shows that $g_{66}$ has exactly one positive root, so $w_1=\theta_{66}$. Irreducibility then gives $K_{66}=K_1$. As $\Delta_1=13$ is prime, the conductor formula gives $R_1=\OO_1=\OO_{66}$. By Thomas's theorem, the classes of $z_1,z_1+1$ form a basis of $\OO_{66}^{\times}/\{\pm1\}$, and the classes of $w_1,w_1+1$ form a basis of $R_{66}^{\times}/\{\pm1\}$. The exponent matrix is
    \[
        \begin{pmatrix}-1&3\\4&1\end{pmatrix},
    \]
    whose determinant is $-13$. Hence $\varepsilon_{66}=13$.
\end{proof}

\section{Ideal classes at prime additive index}
\label{sec:ideal-classes}

Combining the unit-index result in \Cref{cor:prime-index-unit-rigidity} with the local structure makes the change-of-order sequence explicit.

Throughout this section, assume that $N_t$ is a rational prime, and put $p=N_t$. Let $\mathfrak P$ be the unique prime ideal of $\OO_t$ above $N_t$, and put $\mathfrak p=\mathfrak P\cap R_t$. By \Cref{cor:prime-local-structure} and its proof,
\[
    N_t\OO_t=\mathfrak P^3,\qquad R_t=\ZZ+\mathfrak P^2,\qquad \mathfrak p=\mathfrak f_t=\mathfrak P^2.
\]
We prove in \Cref{cor:cardinalities} that $R_t$ is a Bass order and that $\ICM(R_t)$ has the two Picard strata $\Pic(R_t)$ and $\Pic(\OO_t)$.

If $N_t$ is squarefree and composite, or the cube of a rational prime, intermediate orders may occur. These formulas do not extend unchanged to those cases.

\begin{corollary}\label{cor:cardinalities}
    Let $t\in\ZZ$ with $t\geq -1$, and suppose that $N_t$ is a rational prime. Then
    \[
        \ker\!\left(\Pic(R_t)\longrightarrow\Pic(\OO_t)\right)
        \cong
        \begin{cases}
            \{1\},   & (t,N_t)=(3,3)\text{ or }(5,7), \\
            C_{N_t}, & \text{otherwise}.
        \end{cases}
    \]
    Consequently,
    \[
        h(R_t)=
        \begin{cases}
            h(\OO_t),     & (t,N_t)=(3,3)\text{ or }(5,7), \\
            N_t h(\OO_t), & \text{otherwise},
        \end{cases}
    \]
    and
    \[
        \#\ICM(R_t)=
        \begin{cases}
            2h(\OO_t),       & (t,N_t)=(3,3)\text{ or }(5,7), \\
            (N_t+1)h(\OO_t), & \text{otherwise}.
        \end{cases}
    \]
    Here $h(S)=\#\Pic(S)$.
\end{corollary}

\begin{proof}
    By \Cref{cor:prime-local-structure}, one has $\mathfrak f_t=\mathfrak P^2$ and the residue-unit quotient is $C_{N_t}$. The change-of-order exact sequence \cite[Theorem~5.6 and Proposition~6.2]{KluenersPauli2005} therefore specializes to
    \[
        1\longrightarrow R_t^\times\longrightarrow\OO_t^\times
        \longrightarrow C_{N_t}\longrightarrow\Pic(R_t)
        \longrightarrow\Pic(\OO_t)\longrightarrow1.
    \]
    Thus
    \[
        \#\ker\bigl(\Pic(R_t)\to\Pic(\OO_t)\bigr)
        =\frac{N_t}{[\OO_t^\times:R_t^\times]},
    \]
    and the first two assertions follow from \Cref{cor:prime-index-unit-rigidity}.

    Every intermediate order $S$ determines an additive subgroup $S/R_t\subseteq\OO_t/R_t$. The latter group has prime order $N_t$. Thus $S=R_t$ or $S=\OO_t$.

    The monogenic order $R_t$ is Gorenstein by Marseglia's corollary \cite[Corollary~2.12]{Marseglia2020}; the maximal order $\OO_t$ is Dedekind and hence Gorenstein. Since these are the only overorders, $R_t$ is a Bass order. The Bass decomposition \cite[Proposition~3.7]{Marseglia2020} (see also Proposition~6.3 of Cho, Hong, and Lee \cite[Proposition~6.3]{ChoHongLeeBass2025}) therefore takes the form
    \[
        \ICM(R_t)=\Pic(R_t)\sqcup\Pic(\OO_t).
    \]
    Taking cardinalities gives the last formula.
\end{proof}

For comparison, Theorem~5.3(b) of Arpin, Marseglia, and Springer \cite[Theorem~5.3(b), p.~17]{ArpinMarsegliaSpringer2025} uses $\operatorname{Cl}(S)$ for our $\Pic(S)$, and $\mathfrak l$ for the unique prime ideal of $R_t$ containing the conductor. In the present setting,
\[
    \mathfrak l=\mathfrak p=\mathfrak f_t=\mathfrak P^2,
    \qquad
    (\mathfrak l:\mathfrak l)=\OO_t.
\]
Their multiplicator-ring chain is therefore $R_t\subsetneq\OO_t$. The unique prime ideal $\mathfrak P$ above $\mathfrak l$ satisfies $\mathfrak P^2=\mathfrak l$, so this is their ``ramified'' case and $\delta_{\mathfrak l}=0$.

\subsection{Explicit fibers}\label{sec:explicit-fibers}

For a full fractional $R_t$-ideal $\mathfrak a$, write
\[
    (\mathfrak a:\mathfrak a)=\{x\in K_t:x\mathfrak a\subseteq \mathfrak a\}
\]
for its multiplicator ring.

\begin{theorem}\label{thm:explicit-picard-fibers}
    With the same hypotheses as in \Cref{cor:cardinalities}, let
    \[
        \lambda_t:\Pic(R_t)\longrightarrow\Pic(\OO_t),
        \qquad [I]\longmapsto[I\OO_t],
    \]
    be the extension-of-ideals map. Fix an isomorphism of $\FF_{N_t}$-algebras
    \[
        \OO_t/\mathfrak f_t\cong\FF_{N_t}[\overline{\varrho}_p]/(\overline{\varrho}_p^2)
    \]
    under which $R_t/\mathfrak f_t$ is the constant subring $\FF_{N_t}$, and choose $\varrho_p\in\OO_t$ lifting $\overline{\varrho}_p$. For $k\in\FF_{N_t}$, let $\widetilde k\in\{0,\ldots,N_t-1\}$ be its integer representative. Put
    \[
        u_k=1+\widetilde k\varrho_p,\qquad
        I_k=u_kR_t+\mathfrak f_t.
    \]
    Each $I_k$ is a full fractional $R_t$-ideal contained in $\OO_t$, but it need not be contained in $R_t$. Here $k+\ell$ and $-k$ are computed in $\FF_{N_t}$:
    \[
        I_kI_\ell=I_{k+\ell},\qquad
        I_k^{-1}=I_{-k},\qquad
        I_k\OO_t=\OO_t.
    \]
    Moreover,
    \[
        \ker(\lambda_t)=\{[I_k]:k\in\FF_{N_t}\}.
    \]
    The classes $[I_k]$ are distinct unless $(t,N_t)=(3,3)$ or $(5,7)$. In either of these two exceptional cases, every $[I_k]$ is the principal class $[R_t]$.

    For $[\mathfrak a]\in\Pic(\OO_t)$, choose an integral representative $\mathfrak a$ coprime to $\mathfrak f_t$. Put $\mathfrak j=\mathfrak a\cap R_t$. Then $\mathfrak j$ is an invertible $R_t$-ideal, $\mathfrak j\OO_t=\mathfrak a$, and
    \[
        \lambda_t^{-1}([\mathfrak a])
        =
        \{[I_k\mathfrak j]:k\in\FF_{N_t}\}.
    \]
\end{theorem}

\begin{proof}
    The $R_t$-modules $\OO_t/\mathfrak f_t$ and $R_t/\mathfrak f_t$ are supported only at $\mathfrak p$, so their localization maps at $\mathfrak p$ are isomorphisms. The quotient-ring calculation in the proof of \Cref{prop:squarefree-local-structure}, together with $\mathfrak f_t=\mathfrak P^2$, gives the isomorphism used in the statement.

    Changing the lift $\varrho_p$ of $\overline{\varrho}_p$ leaves each $I_k$ unchanged, since it changes $u_k$ only by an element of $\mathfrak f_t$. If $\overline{\varrho}_p$ is replaced by $a\overline{\varrho}_p$, where $a\in\FF_{N_t}^\times$, the new ideal indexed by $k$ is the original $I_{ak}$. Thus the choice of coordinate only relabels the same family of ideals.

    In $\OO_t/\mathfrak f_t$ one has
    \[
        \overline{u_k}\,\overline{u_\ell}
        =(1+k\overline{\varrho}_p)(1+\ell\overline{\varrho}_p)
        =1+(k+\ell)\overline{\varrho}_p
        =\overline{u_{k+\ell}}.
    \]
    Since $\overline{u_k}$ is a unit, $u_k\OO_t+\mathfrak f_t=\OO_t$. Multiplying this ideal equality by $\mathfrak f_t$ gives
    \[
        u_k\mathfrak f_t+\mathfrak f_t^2=\mathfrak f_t.
    \]
    Also $u_ku_\ell-u_{k+\ell}\in\mathfrak f_t$. Therefore
    \[
        \begin{aligned}
            I_kI_\ell
             & =
            u_ku_\ell R_t
            +u_k\mathfrak f_t
            +u_\ell\mathfrak f_t
            +\mathfrak f_t^2                \\
             & =u_ku_\ell R_t+\mathfrak f_t
            =u_{k+\ell}R_t+\mathfrak f_t
            =I_{k+\ell}.
        \end{aligned}
    \]
    Thus $I_0=R_t$, $I_k^{-1}=I_{-k}$, and every $I_k$ is invertible. Extension gives $I_k\OO_t=u_k\OO_t+\mathfrak f_t=\OO_t$.

    The residue classes $1+k\overline{\varrho}_p$ exhaust
    \[
        (\OO_t/\mathfrak f_t)^\times/
        (R_t/\mathfrak f_t)^\times\cong(\FF_{N_t},+).
    \]
    With the contraction convention used in Section~8 of Kl\"uners and Pauli \cite[Section~8, ``Computing Pic($O$)'']{KluenersPauli2005}, the connecting homomorphism
    \[
        \partial_t:\frac{(\OO_t/\mathfrak f_t)^\times}{(R_t/\mathfrak f_t)^\times}\longrightarrow\Pic(R_t)
    \]
    sends the class of $1+k\overline{\varrho}_p$ to the class of $\mathfrak d_k=(u_k\OO_t)\cap R_t$.

    We relate this convention to our ideal representatives by proving $I_k\mathfrak d_k=u_kR_t$ locally. Localization commutes with the contraction defining $\mathfrak d_k$.

    At $\mathfrak p$, the element $u_k$ is a unit in $\OO_{t,\mathfrak P}$, so $(\mathfrak d_k)_{\mathfrak p}=R_{t,\mathfrak p}$. Moreover, $(I_k)_{\mathfrak p}=u_kR_{t,\mathfrak p}$ because $u_k^{-1}(\mathfrak f_t)_{\mathfrak p}=(\mathfrak f_t)_{\mathfrak p}\subseteq R_{t,\mathfrak p}$.

    At every maximal ideal $\mathfrak q\neq\mathfrak p$, the localized orders agree and $(\mathfrak f_t)_{\mathfrak q}=R_{t,\mathfrak q}$. Hence $(\mathfrak d_k)_{\mathfrak q}=u_kR_{t,\mathfrak q}$ and $(I_k)_{\mathfrak q}=R_{t,\mathfrak q}$.

    Since fractional ideals are equal when all their localizations are equal, these local identities give $I_k\mathfrak d_k=u_kR_t$.

    This identity gives $[I_k]=[\mathfrak d_k]^{-1}$. In $\OO_t/\mathfrak f_t$ one has
    \[
        (1+k\overline{\varrho}_p)(1-k\overline{\varrho}_p)=1
    \]
    because $\overline{\varrho}_p^2=0$. Since $\partial_t$ is a homomorphism,
    \[
        \partial_t([1-k\overline{\varrho}_p])
        =\partial_t([1+k\overline{\varrho}_p])^{-1}
        =[\mathfrak d_k]^{-1}
        =[I_k].
    \]
    As $k$ ranges over $\FF_{N_t}$, these classes are precisely the image of $\partial_t$, which is $\ker(\lambda_t)$ by exactness. By \Cref{cor:cardinalities}, this kernel has order $N_t$ in the nonexceptional cases and order one in the exceptional cases.

    Every ideal class of the Dedekind domain $\OO_t$ has an integral representative coprime to the conductor. Choose an integral ideal $\mathfrak b$ in the inverse class $[\mathfrak a]^{-1}$. The module form of the Chinese remainder theorem gives an element $\alpha\in \mathfrak b$ such that
    \[
        \alpha\notin\mathfrak Q \mathfrak b
    \]
    for every prime ideal $\mathfrak Q$ dividing $\mathfrak f_t$. Since $\alpha\in \mathfrak b$, the ideal $(\alpha)\mathfrak b^{-1}$ is integral and belongs to the class $[\mathfrak a]$. For each such $\mathfrak Q$, the choice of $\alpha$ gives
    \[
        v_{\mathfrak Q}((\alpha))=v_{\mathfrak Q}(\mathfrak b).
    \]
    Therefore $\mathfrak Q\nmid(\alpha)\mathfrak b^{-1}$. Thus $(\alpha)\mathfrak b^{-1}$ is coprime to $\mathfrak f_t$.

    For this choice of $\mathfrak a$, we check extension and contraction locally. At every maximal ideal $\mathfrak q\neq\mathfrak p$, the localized orders agree. Thus $\mathfrak j_{\mathfrak q}=\mathfrak a_{\mathfrak q}$ inside their common localization, and this ideal is principal. At the prime ideal $\mathfrak p$ containing the conductor, coprimality gives $\mathfrak a_{\mathfrak P}=\OO_{t,\mathfrak P}$ and hence $\mathfrak j_{\mathfrak p}=R_{t,\mathfrak p}$.

    It follows that $\mathfrak j=\mathfrak a\cap R_t$ is locally principal at every maximal ideal. Marseglia's local-principality criterion \cite[Lemma~2.9]{Marseglia2020} then shows that it is invertible. The same local calculation gives $\mathfrak j\OO_t=\mathfrak a$.

    Every class in $\lambda_t^{-1}([\mathfrak a])$ is a product of $[\mathfrak j]$ and an element of $\ker(\lambda_t)$. Conversely, every such product lies in the fiber, giving the stated classes $[I_k\mathfrak j]$.
\end{proof}

\begin{theorem}\label{thm:explicit-fibers}
    Under the same hypotheses and with the same notation as in \Cref{thm:explicit-picard-fibers}, define
    \[
        \pi_t:\ICM(R_t)\longrightarrow\Pic(\OO_t),\qquad
        [\mathfrak a]\longmapsto[\mathfrak a\OO_t].
    \]
    If $[\mathfrak a]\in\Pic(\OO_t)$ and $\mathfrak j=\mathfrak a\cap R_t$ is chosen as in \Cref{thm:explicit-picard-fibers}, then
    \begin{equation}
        \pi_t^{-1}([\mathfrak a])
        =
        \{[I_k\mathfrak j]:k\in\FF_{N_t}\}
        \sqcup\{[\mathfrak a]_{R_t}\},
        \label{eq:icm-fiber}
    \end{equation}
    where $[\mathfrak a]_{R_t}$ denotes the class of the $\OO_t$-ideal $\mathfrak a$ regarded as an $R_t$-ideal.

    This fiber has cardinality $2$ in the exceptional cases and $N_t+1$ otherwise. The multiplication action of $\ker(\lambda_t)$ on this fiber is free and transitive on its invertible part $\lambda_t^{-1}([\mathfrak a])$ and fixes the remaining class $[\mathfrak a]_{R_t}$. The multiplicator ring of the latter class is $\OO_t$.

    More generally, if $I$ is an invertible fractional $R_t$-ideal and $\mathfrak b$ is a fractional $\OO_t$-ideal, then
    \[
        [I][\mathfrak b]_{R_t}=[(I\OO_t)\mathfrak b]_{R_t}.
    \]
    Multiplication within the $\Pic(\OO_t)$ stratum is the usual multiplication in $\Pic(\OO_t)$.
\end{theorem}

\begin{proof}
    For every full fractional $R_t$-ideal $\mathfrak a$, the extension $\mathfrak a\OO_t$ is a nonzero fractional $\OO_t$-ideal and hence is invertible. If $\mathfrak a'=x\mathfrak a$ for some $x\in K_t^\times$, then $\mathfrak a'\OO_t=x(\mathfrak a\OO_t)$. Thus $\pi_t$ defines a monoid homomorphism on homothety classes.

    The Bass decomposition used in \Cref{cor:cardinalities} shows that $\pi_t$ restricts to $\lambda_t$ on $\Pic(R_t)$ and to the identity on the $\Pic(\OO_t)$ stratum. Therefore \eqref{eq:icm-fiber} follows. Lemmas~2.5 and~3.6 of Marseglia \cite[Lemmas~2.5 and~3.6]{Marseglia2020} show that the union is disjoint: an invertible $R_t$-ideal $I$ has $(I:I)=R_t$, whereas the invertible $\OO_t$-ideal $\mathfrak a$ has $(\mathfrak a:\mathfrak a)=\OO_t$. Since $\lambda_t$ is a surjective group homomorphism, each of its fibers is a torsor under its kernel. Moreover,
    \[
        I_k\mathfrak a=(u_kR_t+\mathfrak f_t)\mathfrak a=(u_k\OO_t+\mathfrak f_t)\mathfrak a=\mathfrak a,
    \]
    so $[\mathfrak a]_{R_t}$ is fixed by the kernel action.

    If $I$ is an invertible fractional $R_t$-ideal and $\mathfrak b$ is a fractional $\OO_t$-ideal, then
    \[
        I\mathfrak b=(I\OO_t)\mathfrak b,
    \]
    because $\mathfrak b$ is an $\OO_t$-module. This proves the mixed-product formula. The assertion about products within the $\Pic(\OO_t)$ stratum follows directly.
\end{proof}

\section{Integral matrix conjugacy}\label{sec:matrix-conjugacy}

The ideal-class fibers give a corresponding classification of integral matrices. After choosing a $\ZZ$-basis of a full $R_t$-lattice, multiplication by $\theta_t$ is represented by an integral matrix with characteristic polynomial $g_t$; its integral conjugacy class is independent of the basis.

\begin{corollary}\label{cor:matrices}
    With the same hypotheses as in \Cref{cor:prime-index-unit-rigidity}, let
    \[
        \mathcal M_t=
        \{M\in M_3(\ZZ):\det(X\operatorname{Id}_3-M)=g_t(X)\}.
    \]
    Then every $M\in\mathcal M_t$ has determinant $1$, and $\GL_3(\ZZ)$-conjugacy and $\SL_3(\ZZ)$-conjugacy define the same equivalence relation on $\mathcal M_t$. Moreover,
    \[
        \#\bigl(\mathcal M_t/\SL_3(\ZZ)\bigr)
        =
        \begin{cases}
            2h(\OO_t),       & (t,N_t)=(3,3)\text{ or }(5,7), \\
            (N_t+1)h(\OO_t), & \text{otherwise}.
        \end{cases}
    \]
    For each $[\mathfrak a]\in\Pic(\OO_t)$, choose an integral representative $\mathfrak a$ coprime to $\mathfrak f_t$ and put $\mathfrak j=\mathfrak a\cap R_t$ as in \Cref{thm:explicit-picard-fibers}. Under the matrix correspondence, the conjugacy classes associated with the invertible part of $\pi_t^{-1}([\mathfrak a])$ are indexed by the distinct classes among
    \[
        \{[I_k\mathfrak j]:k\in\FF_{N_t}\}.
    \]
    The remaining class corresponds to $\mathfrak a$ as a noninvertible $R_t$-ideal.
\end{corollary}

\begin{proof}
    The polynomial $g_t$ is irreducible, hence squarefree, and evaluation at $\theta_t$ induces $R_t\cong\ZZ[X]/(g_t)$. For $M\in\mathcal M_t$, its minimal polynomial is a nonconstant divisor of the irreducible $g_t$, hence equals $g_t$. The generalized Latimer--MacDuffee correspondence \cite[Corollary~8.2(a)]{Marseglia2020} therefore identifies $\ICM(R_t)$ with the $\GL_3(\ZZ)$-conjugacy classes in $\mathcal M_t$; see also Conrad's theorem \cite[Theorem~2.1, pp.~2--5]{ConradIdealClasses}.

    On a full $R_t$-lattice $\mathfrak a\subset K_t$, a $\ZZ$-basis of $\mathfrak a$ gives the multiplication matrix of $\theta_t$. Its determinant is $\Norm_{K_t/\QQ}(\theta_t)=1$, since the constant term of $g_t$ is $-1$.

    If a conjugating matrix $C\in\GL_3(\ZZ)$ has determinant $-1$, then
    \[
        \det(-C)=(-1)^3\det(C)=1,
    \]
    and $-C$ induces the same conjugation. Thus the two conjugacy relations coincide.

    The count follows from \Cref{cor:cardinalities}, and the fibers from \Cref{thm:explicit-fibers}. Here invertible and noninvertible refer to $R_t$-ideals.
\end{proof}

Wallace's correspondence \cite[Theorem~4, pp.~182--183]{Wallace1984} applies because $g_t$ is irreducible with three distinct real roots and constant term $-1$. Under her definition \cite[Definition~2, p.~181]{Wallace1984}, ``narrow ideal classes'' are homothety classes modulo scalars of positive $K_t/\QQ$-norm. The standard narrow class group uses totally positive scalars.

Since $[K_t:\QQ]=3$, if $\mathfrak j=xI$ and $\Norm_{K_t/\QQ}(x)<0$, then $\mathfrak j=(-x)I$ and $\Norm_{K_t/\QQ}(-x)>0$. This is Wallace's odd-degree corollary \cite[Corollary~1, p.~183]{Wallace1984}: her narrow and wide classes coincide in this setting. Conrad's remark \cite[Remark~2.7]{ConradIdealClasses} contains the same odd-degree sign argument.

\bibliographystyle{alpha}
\bibliography{references}

\end{document}